\documentclass[12pt]{amsart} %oneside
\usepackage[margin=1.2in]{geometry}
\usepackage{graphicx}
\usepackage{bm}
\usepackage{multirow}
\usepackage[table,xcdraw]{xcolor}
\usepackage{graphicx}
\usepackage[authoryear, round]{natbib}
\usepackage{amsmath}
\usepackage{placeins}
\usepackage{amssymb}
\usepackage{epstopdf}
\usepackage{pdfpages}
\usepackage{amsthm} 
\usepackage{xcolor}
\usepackage{setspace}
\usepackage{dsfont}

\newtheorem{theorem}{Theorem}[]
\newtheorem{lemma}[]{Lemma}

\newtheorem{example}[]{Example}[]
\newtheorem{definition}[]{Definition}
\newtheorem{remark}[]{Remark}

\newcommand*\E{\mathop{}\!\mathbb{E}}
\newcommand{\Exp}{\mathds{E}}

\newcommand{\ii}{\mbox{i}}
\newcommand{\Prob}{\mathop{}\!\mathbb{P}}

\newcommand{\vect}[1]{\vec{#1}}

\renewcommand{\vect}[1]{\boldsymbol{#1}}

\newcommand{\mat}[1]{\boldsymbol{#1}}

\newcommand*\dd{\mathop{}\!\mathrm{d}}

\usepackage{multicol}

\author[H. \smash{Albrecher}]{Hansj\"org Albrecher$^1$}
\address[Hansj\"org Albrecher]{ Department of Actuarial Science, Faculty of Business and Economics and Swiss Finance Institute, University of Lausanne, CH-1015 Lausanne, Switzerland}

\email{{hansjoerg.albrecher@unil.ch}}

\author[M. \smash{Bladt}]{Martin Bladt}
\address[Martin Bladt]{Department of Actuarial Science, Faculty of Business and Economics, University of Lausanne, CH-1015 Lausanne, Switzerland}

\email{{martin.bladt@unil.ch}}

\author[M. \smash{Bladt}]{Mogens Bladt}
 \address[Mogens \smash{Bladt}]{Department of Mathematical Sciences, University of Copenhagen, Universitetsparken 5, DK-2100 Copenhagen \O, Denmark}
\email{bladt@math.ku.dk}

\title{Multivariate Matrix Mittag--Leffler distributions}

\begin{document}
\footnote{Corresponding author.}

%\tableofcontents
\begin{abstract}
We extend the construction principle of multivariate phase-type distributions to establish an analytically tractable class of heavy-tailed multivariate random variables whose marginal distributions are of Mittag-Leffler type with arbitrary index of regular variation. The construction can essentially be seen as allowing a scalar parameter to become matrix-valued. The class of distributions is shown to be dense among all multivariate positive random variables and hence provides a versatile candidate for the modelling of heavy-tailed, but tail-independent, risks in various fields of application. 
\end{abstract} 
\maketitle
\section{Introduction}
The joint modelling of dependent risks is a crucial task in many areas of applied probability and quantitative risk management, see e.g.\ \citet{McN}. While in many situations there is a reasonable amount of data available for the fitting procedure of univariate risks, the identification of {multivariate models} is much more delicate. A frequent approach proposed in applications is to use the available data for univariate fitting, and choose a parametric copula to combine the margins, where the parameters of that copula are then either assumed a priori or estimated from the joint data. The choice of such a copula is of course crucial for the resulting joint distribution and the conclusions one draws from it, cf.\ \citet{Scherer,Mikosch}. In multivariate extremes, which is currently a very active research topic, one typically uses less restrictive assumptions for the quantification of joint exceedances, see e.g.\ \citet{Falk,Kiri}. Some specific families, like multivariate regular variation, are considered particularly attractive in this context, as they have a natural interpretation in terms of how to extend univariate behaviour into higher dimensions \citet{Dombry,joe,resnick02}. These results focus, however, on the asymptotic behaviour, so that for a concrete application with an available data set one typically has to choose thres\-holds above which this respective behaviour is assumed \citet{wan}, and the bulk of the distribution {is} then to be modelled by a different distribution (see e.g.\ \citet{beirlant2006} and \cite[Ch.IV.5]{abt}).  \\

In this paper we would like to establish a family of multivariate distributions that can be applied for modeling across the entire positive orthant, so that no thres\-hold selection is needed. In particular, we are interested in a family that leads to explicit and tractable expressions for the model fitting and interpretation. While such a family already exists for marginally light (exponentially bounded) tails in the form of multivariate phase-type (MVPH) distributions, our goal here is to develop a related family with heavy-tailed marginal distributions. The univariate starting point for this procedure is the matrix Mittag-Leffler (MML) distribution, which is a heavy-tailed distribution that was recently studied in \citet{albrecher2019matrix}, and which proved to be very tractable, with excellent fitting properties. While in principle there are many possible ways of defining a vector of random variables with given marginals, we want to consider here the natural concept of multivariate families that can be characterized by the property that any linear combination of the components of such a vector is again of the same marginal type. This is exactly one possible definition of MVPH distributions (so any linear combination of the coordinates of a random vector are again (univariate) phase-type), and it is also a characterizing property of multivariate regular variation of a random vector, namely that any linear combination of the coordinates of such a vector is again (univariate) regularly varying, see \citet{basrak}. \\

The goal is hence to study the class of multivariate random vectors for which such a property applies with MML marginal distributions. It will turn out that for this approach to work, we first need to consider a slightly more general class, which we will refer to as generalized MML distributions. We will show that the analysis developed for the MVPH case can then be extended to our more general situation. In particular, we will establish some properties of this class and work out explicit expressions for a number of concrete cases. The analysis is considerably simpler for the symmetric situation where all marginal distributions share the same index of regular variation, but the general case can be handled as well. The resulting multivariate MML distribution is asymptotically independent, i.e.\ there is tail-independence for each bivariate pair of components. In the case of multivariate regular variation, the subclass of random vectors with asymptotic independence was studied and characterized in terms of second order conditions in \citet{resnick02}, where also concrete application areas for such heavy-tailed, but asymptotically independent risks are given. In a sense, the multivariate MML family of distributions we introduce here is another candidate for models in this domain, with the advantage of being explicit and tractable across the entire range ${\mathbb R}_+^n$. In that respect, this family is also an interesting alternative to multivariate Linnik distributions (see e.g.\ \citet{anderson} and \citet{lim}), which can be conveniently defined in terms of their characteristic function, have the range ${\mathbb R}^n$ (rather than ${\mathbb R}_+^n$) and also have  heavy-tailed marginals, but which do not lead to explicit expressions for the multivariate density. \\

The remainder of the paper is organized as follows. Section \ref{sec2} recapitulates the construction principle of univariate and multivariate PH distributions and provides the available background on MML distributions. Section \ref{sec:GMML} introduces generalized MML distributions. In 
Section \ref{sec4} we then develop the necessary theoretical background for our definition of the multivariate MML family and establish some of its properties. We also consider power transforms, which will provide useful flexibility for modeling applications, and we derive denseness properties of the resulting multivariate family. In Section \ref{sec:examples} we work out a concrete simple example in detail and illustrate resulting dependence properties for this case. Section \ref{sec:concl} concludes.
%--------------------------
%In this paper we introduce a multivariate theory for Mittag--Leffler (ML) type of distributions. When establishing a multivariate theory there is a number of questions which inevitably will have to be addressed such as marginal distributions and linear combinations (projections) of the coordinate vectors. Concerning the latter point, convolutions of independent ML distributions with the same tail index is not a ML distribution so projections in the different directions would no longer belong to this class. They are, however, contained in a larger class of so--called Matrix Mittag--Leffler (MML) distributions which was introduced in \citet{albrecher2019matrix}. The construction therein is based on Phase--type distributions, which makes the class both versatile and tractable. The MML class, which is closed under finite convolutions, will for the basis for defining multivariate MML with the same tail index in all coordinate directions as those for which any linear combination of the coordinate random variables have a MML. If we allow for different tail indices for the different coordinates then it will be necessary to introduce a larger class of MML distributions, referred to as generalized MML or GMML, which will suffice to describe the different projections or linear combinations of coordinate processes. 

\section{Phase--type distributions}\label{sec2}
\subsection{Notation}
We shall apply a common convention from phase--type theory that matrices are expressed in bold capital letters (e.g. $\mat{T}, \mat{\Lambda})$, row vectors are bold minuscular greek letters (e.g. $\vect{\pi}, \vect{\alpha})$ while column vectors are bold minuscular roman letters (e.g. $\vect{t}$, $\vect{x})$. Elements of matrices and vectors are denoted by their corresponding minuscular unbold letters with indices, e.g. $\mat{A}=\{ a_{ij}\}$ and 
 $\vect{a}=(a_i)$. If $\vect{a}=(a_1,...,a_n)$ is a vector, then by $\mat{\Delta}(\vect{a})$ we shall denote the diagonal matrix with $\vect{a}$ as diagonal. 
\subsection{Univariate phase--type distributions}
Phase--type distributions are defined as the distribution of the time until absorption 
of a finite state--space Markov jump process with one absorbing state and the other states being transient. 

{ Let $p$ be a positive integer, and} $\{ X_t \}_{t\geq 0}$ denote a Markov jump process on $E=\{1,...,p,p+1\}$, where states $1,2,...,p$ are transient and state $p+1$ is absorbing. Let $\pi_i = \Prob (X_0=i)$ and assume that $\pi_1+\cdots \pi_p =1$, i.e. initiation in the absorbing state is not possible. The intensity matrix of $\{ X_t\}_{t\geq 0}$ can be written as 
 \begin{equation}
 \mat{\Lambda} = 
 \begin{pmatrix}
 \mat{T} & \vect{t} \\
 \vect{0} & 0 
 \end{pmatrix} , \label{basic-full-int-mat}
 \end{equation}
  where $\mat{T}$ is the $p\times p$ {\it sub--intensity} matrix whose off diagonal elements consist of transition rates between the transient states, $\vect{t}$ is a $p$--dimensional column vector 
$\vect{0}$ is a $p$--dimensional row vector. { The diagonal elements of $\mat{T}$ are given by $t_{ii}=-\sum_{j\neq i}t_{ij} + t_i$, since the row sums of $\mat{\Lambda}$ must be zero.}

 Let $\vect{e}$ denote the vector of ones and $\vect{\pi}=(\pi_1,...,\pi_p)$. Dimensions are usually suppressed and $\vect{e}$ may then have any adequate dimension depending on the context.

 Then the time until absorption,
 \[   \tau = \inf\{ t\geq 0 : X_t=p+1 \} , \]
 is said to have a phase--type (PH) distribution with representation $(\vect{\pi},\mat{T})$ and we write $\mbox{PH}_p(\vect{\pi},\mat{T})$. Since rows of $\mat{\Lambda}$ sum to zero, we get $\vect{t}=-\mat{T}\vect{e}$. { Note that the case $p=1$ leads to an exponential distribution.}

 If $\tau \sim \mbox{PH}_p(\vect{\pi},\mat{T})$, then a number of relevant formulas can be written compactly in matrix notation, like e.g.
 \begin{eqnarray*}
 f(x;{ \vect{\pi},\mat{T}})&=&\vect{\pi}e^{\mat{T}x}\vect{t}, \quad { x>0,}\\
 F(x;{ \vect{\pi},\mat{T}})&=&1-\vect{\pi}e^{\mat{T}x}\vect{t}, \quad { x>0,}\\
 L(s;{ \vect{\pi},\mat{T}})&=&\vect{\pi}(s\mat{I}-\mat{T})^{-1}\vect{t}, \quad { s>\mbox{\rm Re}(\eta_{\rm \max}),}\\
 \Exp (\tau^\alpha)&=& \Gamma (\alpha +1)\vect{\pi}(-\mat{T})^{-\alpha}\vect{e}, \quad { \alpha>0,}
 \end{eqnarray*}
for the density, c.d.f., Laplace transform and (fractional) moments, respectively. { Here $\eta_{\rm max}$ denotes the eigenvalue with maximum real part of $\mat{T}$, and this real part is strictly negative. In particular, the Laplace transform is well defined for all $s\geq 0$ and in a neighbourhood around zero. }

\begin{remark}\rm
{ 
Representations $(\vect{\pi},\mat{T})$ of phase--type distributions are not unique. In fact, one can construct an infinite number of different representations, which may even be of different orders $p$. Hence phase--type representations may also suffer from over-parametrisation, and it is not possible to attach a specific significance to individual elements of an intensity matrix. 
{
While one can typically construct a certain behaviour  by means of structuring the sub--intensity matrix $\mat{T}$, the opposite task of deducing such a behaviour from a given matrix is typically not possible. Some simple cases, however, may be described. For instance, $p=1$ means one phase and the resulting distribution is exponential, hence unimodal. For $p=2$, bimodality cannot be achieved either, as one could at most aim for a mixture of exponentials. For $p=3$ it is possible to have a mixture of an exponential with an Erlang($2$) which is bimodal. }
}
\end{remark}
For further details on phase--type expressions, we refer to \citet{albrecher2019matrix} and \citet{bladt2017matrix}.

\subsection{Multivariate phase--type distributions}
A non--negative random vector $\vect{X}=(X_1,...,X_n)$ is phase--type distributed (MVPH) if all non--negative, non-vanishing linear combinations of its coordinates $X_i$, $i=1,...,n$ have a (univariate) phase--type distribution. This is the most general definition of a multivariate phase--type distribution which, however, lacks practicality since it does not suggest how to construct such distributions. It contains a sub--class of multivariate distributions, MPH$^*$, which have multidimensional Laplace transforms of the form
 \begin{equation}
  L_{\mat{X}}(\vect{u}; { \vect{\pi},\mat{T},\mat{R}} ) =\Exp(e^{-<\vect{u},\mat{X}>})= \vect{\pi} \left( \mat{\Delta}(\mat{R}\vect{u}) -\mat{T} \right)^{-1}\vect{t} .\label{eq:Kulkarni} 
  \end{equation}
and we write that $\vect{X}\sim \mbox{MPH}^*(\vect{\pi},\mat{T},\mat{R})$.  
Here $(\vect{\pi},\mat{T})$ is a phase--type representation of dimension $p$, say, $\mat{R}$ is a $p\times n$ matrix and $\vect{u}=(u_1,...,u_n)$ {$\in \mathbb{R}_+^n$. Furthermore, the joint Laplace transform exists in a neighbourhood around zero (\cite[Thm.8.1.2]{bladt2017matrix})}.

The form \eqref{eq:Kulkarni} is established from the following probabilistic construction (cf.\ \citet{Kulkarni:1989ti}). Consider the Markov jump process $\{ X_t\}_{t\geq 0}$ underlying the phase--type distribution with representation $(\vect{\pi},\mat{T})$. The $n$ columns of $\mat{R}=\{ r_{ik} \}$ are $p$--dimensional vectors which contain non--negative numbers. These numbers are ``rewards'' to be earned during sojourns in state $i$. If 
$\tau$ denotes the time until absorption of the underlying Markov jump process, then
\begin{equation}
  X_k = \int_0^\tau \sum_{i=1}^p1\{ X_t=i \} r_{ik}\dd t,\ \ \ k=1,...,n   \label{reward-construction}
\end{equation} 
is the total reward earned according to column $k$ of $\mat{R}$ until absorption. { The structure matrix $\mat{R}$ hence picks scaled sojourns out of the underlying Markov jump process. Correlation between different total rewards, $X_i$ and $X_j$ say, will then depend on the structure of $\mat{R}$ and on the underlying stochastic process. If there are common states in which reward is earned for both $X_i$ and $X_j$, then this will contribute to a positive correlation between them. If there are no common states, the correlation will be entirely generated by the structure of the $\mat{T}$ matrix. Negative correlation between $X_i$ and $X_j$ is achieved if large rewards earned in one reduces the one earned in the other and vice versa. Specific constructions of dependencies between Phase--type distributed random variables with given marginals is non--trivial, see. e.g.\ \citet{biv1} for an example with exponentially distributed marginals. %we write the exponential distributions as a higher order Phase--type distributions on an augmented state--space.
}

{The random variables $X_k$ defined in \eqref{reward-construction}} are again 
phase--type distributed and in general dependent since different variables may be generated through earning positive rewards on certain common states (while in other states there may be zero reward for one variable whenever the other has positive reward). 
If all $r_{ik}> 0$, $i=1,...,p$, then $X_k$ is phase--type distributed with initial distribution $\vect{\pi}$ and 
sub--intensity matrix $\mat{\Delta}^{-1}(\vect{r}_{\cdot k})\mat{T}$. This follows easily from a sample path argument: if reward $r_{ik}$ is earned during a sojourn in state $i$, then the distribution of the reward during a sojourn is exponentially distributed with intensity $-t_{ii}/r_{ik}$. 

If some $r_{ik}=0$, then finding a representation for $X_k$ is more involved. Let $\vect{w}\geq \vect{0}$ denote a non--zero vector. For obtaining the $k$'th marginal distribution we would choose $\vect{w}=\vect{e}_k^\prime$, the $k$'th Euclidean unit vector, while for a more general projection we may
choose $\vect{w}=c_1 \vect{e}_1+...+c_n \vect{e}_k$ for some constants $c_i$, $i=1,...,n$. For this given $\vect{w}$, decompose the set of transient states $E=\{1,...,p\}$ into $E=E_+ \cup E_0$, where $E_+$ denotes states $i\in E$ for
which $(\mat{R}\vect{w})_i >0$ and $E_0$ states $i\in E$ for
which $(\mat{R}\vect{w})_i = 0$. Decompose $\vect{\pi}=(\vect{\pi}_+,\vect{\pi}_0)$ and 
\begin{equation}
  \mat{T} = \begin{pmatrix}
\mat{T}_{++} & \mat{T}_{+0} \\
\mat{T}_{0+} & \mat{T}_{00}
\end{pmatrix}\label{eq:decompose} \end{equation}
accordingly. Then we have the following theorem which is proved in \cite[p.441]{bladt2017matrix}. 
\begin{theorem}\label{th:projection-MPH}
The distribution of $\langle \vect{X},\vect{w}\rangle$ is given by an atom at zero of size ${ q}=\boldsymbol{\pi}_{0}\left(\boldsymbol{I}-\left(-\boldsymbol{T}_{00}\right)^{-1} \boldsymbol{T}_{0+}\right) \boldsymbol{e}$ and an absolute continuous part given by a possibly defective phase-type distribution with representation $(\vect{\pi}_{\vect{w}},\mat{T}_{\vect{w}})$, where
\[ \boldsymbol{\pi}_{w}=\boldsymbol{\pi}_{+}+\boldsymbol{\pi}_{0}\left(-\boldsymbol{T}_{00}\right)^{-1} \boldsymbol{T}_{0+} \text { and } \boldsymbol{T}_{w}=\boldsymbol{\Delta}\left((\boldsymbol{R} \boldsymbol{w})_{+}\right)^{-1}\left(\boldsymbol{T}_{++}+\boldsymbol{T}_{+0}\left(-\boldsymbol{T}_{00}\right)^{-1} \boldsymbol{T}_{0+}\right)  \]
\end{theorem}
This means that 
\begin{eqnarray}
\vect{\pi} \left(  \mat{\Delta}(\mat{R}u\vect{w}) -\mat{T} \right)^{-1}\vect{t}&=&\Exp \left( e^{-\langle \vect{X},u\vect{w}\rangle} \right) \nonumber \\
&=&\Exp \left( e^{-u\langle \vect{X},\vect{w}\rangle} \right) \nonumber \\
&=& { q} + \vect{\pi}_{\vect{w}}(u\mat{I}-\mat{T}_{\vect{w}})^{-1}\vect{t}_{\vect{w}}, \label{eq:basic-id}
\end{eqnarray}
where $\vect{t}_{\vect{w}}=-\mat{T}_{\vect{w}}\vect{e}$.

\begin{remark}\normalfont 
It is still an open question whether $\mbox{MPH}^*\!\subset\!\mbox{MVPH}$ or whether $\mbox{MPH}^*=\mbox{MVPH}$.
%the class of distributions with a Laplace transform of the form \eqref{eq:Kulkarni} is a proper subset of the class of multivariate phase--type distributions.
\end{remark}

\begin{remark}\rm
{ As for univariate phase--type distributions, representations $(\vect{\pi},\mat{T},\mat{R})$ of MPH$^*$ are not uniquely determined by their distributions, and they may be over--parametrised as well. In particular, the interplay between $\mat{T}$ and $\mat{R}$ introduces further ambiguity.}
\end{remark}

While both MPH$^*$ and MPVH distributions lack explicit formulas for distribution and density functions, there is a sub--class of MPH$^*$ distributions that does allow explicit forms. The latter is the one where the structure of the underlying Markov chain is of so--called {\it feed--forward} type. 

Let $\mat{C}_1,...,\mat{C}_n$ be sub--intensity matrices and let $\mat{D}_1,...,\mat{D}_n$ denote non--negative matrices 
such that $-\mat{C}_i\vect{e}=\mat{D}_i\vect{e}$. The matrices $\mat{D}_i$ are not necessarily square matrices, with the number of rows being equal to the number of rows in $\mat{C}_i$ and the number of columns equal to the number of rows (and columns) of $\mat{C}_{i+1}$. Define
\begin{align}\label{feed-forward_structure1} 
\vect{\beta}=(\vect{\pi},\vect{0},...,\vect{0}) \ \ \mbox{and} \ \ \mat{T}=
\begin{pmatrix}
\mat{C}_1 & \mat{D}_1 & \mat{0} & \cdots & \mat{0} \\
\mat{0} &  \mat{C}_2 & \mat{D}_2 & \cdots & \mat{0} \\
\mat{0} & \mat{0} & \mat{C}_3 & \cdots & \mat{0} \\
\vdots & \vdots & \vdots & \vdots\vdots\vdots & \vdots \\
\mat{0} & \mat{0} &\mat{0} & \cdots & \mat{C}_n
\end{pmatrix} 
 \end{align}
 and let the reward matrix be
 \begin{align}\label{feed-forward_structure2} 
 \mat{R}= \begin{pmatrix}
 \vect{e} & \vect{0} & \vect{0} & \cdots & \vect{0} \\
 \vect{0} & \vect{e} & \vect{0} & \cdots & \vect{0} \\
 \vect{0} & \vect{0} & \vect{e} & \cdots & \vect{0} \\
 \vdots & \vdots & \vdots & \vdots\vdots\vdots & \vdots \\
 \vect{0} & \vect{0} & \vect{0} & \cdots & \vect{e}
 \end{pmatrix} . 
 \end{align}
 {The structure of the $\mat{R}$ matrix implies that the $i$'th total reward, $X_i$, then equals the inter--arrival time between arrivals $i-1$ and $i$. 
Positive correlation between two consecutive inter--arrivals $i-1$ and $i$ can then be obtained by choosing the matrix $\mat{D}_i$ in such a way that a long (short) duration of the Markov chain in block $i-1$ will imply a long (short) duration in block $i$ as well. For a negative correlation we have to choose the matrix $\mat{D}_1$ such that the implications are reversed.
 } 
The joint density of the MPH$^*$ distribution is {then} given by 
\begin{equation}
f(x_1,...,x_n;{ \vect{\beta},\mat{T},\mat{R}})= \vect{\pi}e^{\mat{C}_1x_1}\mat{D}_1 e^{\mat{C}_2x_2}\mat{D}_2\cdots \mat{D}_{n-1}e^{\mat{C}_n x_n}\mat{D}_n\vect{e}.\label{eq:dens-map}
\end{equation}
\begin{remark}\rm
The matrices $\mat{C}_i$ are sub--intensity matrices, providing a phase--type distributed
time until arrival $i$. The matrices $\mat{D}_i$ are non--negative matrices containing intensities for initiating a new inter--arrival time for arrival $i+1$ at the time of the arrival $i$. Hence the matrices $\mat{D}_i$ create the dependence between the inter--arrivals. In particular, if $\mat{D}_i = \vect{c}_i\vect{\pi}_{i+1}$, where $\vect{c}_i=-\mat{C}_i\vect{e}$ is the exit rate (column) vector corresponding to $\mat{C}_i$ and $\vect{\pi}_{i+1}$ is some probability (row) vector on $\{1,2,...,p_i\}$, then the inter--arrivals are independent.
\end{remark}

\begin{remark}\rm
{ The (full) matrix $\mat{D}_n$ is not really needed for our purposes, but only the exit vector $\vect{c}_n=-\mat{C}_n\vect{e}=\mat{D}_n\vect{e}$. Thus we may rewrite \eqref{eq:dens-map} in the form}
 
\begin{equation}
f(x_1,...,x_n;{ \vect{\beta},\mat{T},\mat{R}})= \vect{\pi}e^{\mat{C}_1x_1}\mat{D}_1 e^{\mat{C}_2x_2}\mat{D}_2\cdots \mat{D}_{n-1}e^{\mat{C}_n x_n}\vect{c}_n. \label{eq:dens-map1}
\end{equation}
We shall, however, maintain the notation with $\mat{D}_n$ for notational reasons. {Since 
$-\mat{C}_i\vect{e}=\mat{D}_i\vect{e}$ for all $i$, this also implies the exit vector
\[  \vect{t} = -\mat{T}\vect{e} = (0,0,...,0,\vect{c}_n)^\prime  , \]
so $\mat{D}_n\vect{e}$, which is not part of $\mat{T}$, is part of $\vect{t}$ (see \eqref{basic-full-int-mat}).
}
\end{remark}

\begin{remark}\rm
{ Note that the restriction $-\mat{C}_i\vect{e}=\mat{D}_i\vect{e}$ reduces the effective number of parameters contributed from those matrices from $2p_i^2$ to $2p_i^2-p_i$. In particular, the model of \eqref{eq:dens-map1}, and therefore also  \eqref{eq:dens-map}, has $p_1-1+\sum_{i=1}^{n-1}p_i(2p_i-1)+p_n^2$ effective degrees of freedom.}

\end{remark}

\begin{remark}\rm
If $\mat{C}_i=\mat{C}$ and $\mat{D}_i=\mat{D}$ for all $i$, then \eqref{eq:dens-map} is the joint density function for the first $n$ inter--arrival times of a Markovian Arrival Process (MAP) (see e.g.\  \citet{Neuts:1979tp},\citet{bladt2017matrix}). This class of point processes is dense in class of point process on $\mathbb{R}_+$ (see \citet{Asmussen:1993tn}), and therefore the class of distributions given by \eqref{eq:dens-map} is also dense {-- in the sense of weak convergence and with flexible dimension of the matrices $\mat{C}$ and $\mat{D}$ --} in the class of multivariate distributions on $\mathbb{R}_+^n$. 
\end{remark}
Later we shall need the joint fractional moments for such distributions, which are given in the following lemma.
\begin{lemma}\label{lemma:MAP-structure}
Suppose that $\vect{X}=(X_1,X_2,...,X_n)$ has a joint phase--type distribution with density  \eqref{eq:dens-map}.
Then { for $\theta_i>0$, $i=1,\dots,n$},
\[  \Exp (X_1^{\theta_1}X_2^{\theta_2}\cdots X_n^{\theta_n}) = \left(\prod_{i=1}^n\Gamma (\theta_i+1)\right)
\vect{\pi}\left( \prod_{i=1}^n  (-\mat{C}_i)^{-\theta_i-1}\mat{D}_i \right)\vect{e}
  \]
\end{lemma}
\begin{proof}
It is sufficient to prove the lemma for $n=2$.
\begin{eqnarray*}
 \Exp (Z_1^{\theta_1}Z_2^{\theta_2})&=&\int_0^\infty\int_0^\infty z_1^{\theta_1}z_2^{\theta_2}
 \vect{\pi}e^{\mat{C}_1z_1}\mat{D}_1e^{\mat{C}_2z_2}\mat{D}_2\vect{e}\dd z_1 \dd z_2\\
 &=& \vect{\pi}\int_0^\infty z_1^{\theta_1}e^{\mat{C}_1z_1}\dd z_1 \mat{D}_1 
\int_0^\infty z_2^{\theta_2}e^{\mat{C}_2z_2}\dd z_2 \mat{D}_2\vect{e} \\
&=& \vect{\pi} L_{z^{\theta_1}}(-\mat{C}_1) \mat{D}_1 L_{z^{\theta_2}}(-\mat{C}_2) \mat{D}_2\vect{e} , 
\end{eqnarray*}
where $L_{z^{\theta}}(u) = \Gamma (u+1)/u^{\theta+1}$ is the Laplace transform for $z\rightarrow z^\theta$. Since the Laplace transforms are analytic (where they are defined), the result follows by a functional calculus argument (see Theorem 3.4.4 of \citet{bladt2017matrix}).
\end{proof}

\subsection{Matrix Mittag--Leffler distributions}
Let $(\vect{\pi},\mat{T})$ be a phase--type representation. Then a random variable $X$ has a matrix Mittag--Leffler (MML) distribution with representation $(\alpha,\vect{\pi},\mat{T})$, if it has Laplace transform
\[  L_X(u;{ \alpha,\vect{\pi},\mat{T}})=\vect{\pi}\left( u^\alpha \mat{I} - \mat{T}  \right)^{-1}\vect{t} , \quad { u\ge 0,}\]
where $0<\alpha\leq 1$. We write $X\sim \mbox{MML}(\alpha,\vect{\pi},\mat{T})$. Let
\[ E_{\alpha, \beta}(z)=\sum_{k=0}^{\infty} \frac{z^{k}}{\Gamma(\alpha k+\beta)}, \quad {z\in\mathbb{C}}, \]
denote the Mittag--Leffler (ML) function. Then (see \citet{albrecher2019matrix}) the density of $X$ is given by
\[  f(x;{\alpha,\vect{\pi},\mat{T}})=x^{\alpha-1}\vect{\pi}\,E_{\alpha,\alpha}\left(\mat{T}x^\alpha \right)\,\vect{t},\quad {x>0},    \]
and the corresponding c.d.f.\ is
\[  F(x;{\alpha,\vect{\pi},\mat{T}})=1-\vect{\pi}E_{\alpha,1}\left(\mat{T}x^\alpha \right)\vect{e}, \quad {x>0}.  \]
The ML function with { (complex)} matrix argument $\vect{A}$ is defined as
\begin{align*}
E_{\alpha, \beta}(\vect{A})=\sum_{k=0}^{\infty} \frac{\vect{A}^{k}}{\Gamma(\alpha k+\beta)} .
\end{align*}
For $\beta>0$, one can express the (then entire) ML function of a matrix $\mat{A}$ by Cauchy's formula
\[ E_{\alpha,\beta}(\mat{A}) =\frac{1}{2\pi \ii}\int_\gamma E_{\alpha,\beta}(z)(z\mat{I}-\mat{A})^{-1}\dd z , \]
where $\gamma$ is a simple path enclosing the eigenvalues of $\mat{A}$. Invoking the residue theorem, for each entry of the matrix $E_{\alpha,\beta}(z)(z\mat{I}-\mat{A})^{-1}$, then provides a simple method for calculating $E_{\alpha,\beta}(\mat{A})$. 

As outlined in \citet{albrecher2019matrix}, MML distributions { with $0<\alpha<1$} are heavy-tailed with tail indices less than one, so that their mean does not exist. This may be too restrictive in many situations, and one way to obtain a closely related class of distributions is by considering power transformations of the original MML distributed random variables. Indeed, if $X\sim \mbox{MML}(\alpha,\vect{\pi},\mat{T})$,  then $X^{1/\nu} $ has density
\[  f(x;{ \nu,\alpha,\vect{\pi},\mat{T}})=\nu x^{\nu\alpha-1}\vect{\pi}E_{\alpha,\alpha}\left(\mat{T}x^{\nu\alpha} \right)\vect{t},\quad {x>0}, \]
and distribution function
\[  F(x;{ \nu,\alpha,\vect{\pi},\mat{T}})=1-\vect{\pi}E_{\alpha,1}(\mat{T}x^{\alpha \nu})\vect{e},\quad {x>0}, \]
for $\nu>0$ (cf.\ \citet{albrecher2019matrix}).  Rewriting $\beta = \nu \alpha$ leads to { the reparametrization}
\begin{equation}
  f(x;{ \beta,\alpha,\vect{\pi},\mat{T}})=\frac{\beta}{\alpha} x^{\beta-1}\vect{\pi}E_{\alpha,\alpha}\left(\mat{T}x^{\beta} \right)\vect{t} ,\quad {x>0},\label{eq:power-MML-dens}
   \end{equation}
   and
\begin{equation}
   F(x;{ \beta,\alpha,\vect{\pi},\mat{T}})=1-\vect{\pi}E_{\alpha,1}(\mat{T}x^{\beta})\vect{e},\quad {x>0} . \label{eq:power-MML-cdf}
\end{equation}
Thus, for any $0<\alpha\leq 1$ and $\beta>0$, \eqref{eq:power-MML-dens} and \eqref{eq:power-MML-cdf} define densities and their corresponding distribution functions, with tail index $\beta$ instead of $\alpha$. We shall refer to distributions with densities of the form \eqref{eq:power-MML-dens} as power MML and write $X\sim \mbox{MML}^{1/\nu}(\alpha,\vect{\pi},\mat{T})$. Their Laplace transforms are somewhat more involved. 
Indeed, the Laplace transform for $X\sim \mbox{MML}^{1/\nu}(\alpha,\vect{\pi},\mat{T})$ is given by (see formula (5.1.30) in \citet{ML-book-2014} and compare to \cite[p.364]{ML-book-2014})
%define the function
% \[  
% H(z) = \frac{1}{2\pi\ii}\int_\gamma \frac{\Gamma (s)\Gamma (1-s)\Gamma (\beta (1-s))}{\Gamma (\alpha (1-s))}z^{-s}ds ,
%   \] 

 \begin{equation}
 {L_X(s; \nu,\alpha,\vect{\pi},\mat{T})=s^{-\nu\alpha} \vect{\pi} \left( \sum_{k=0}^\infty  \frac{\Gamma (\nu\alpha (k+1))}{\Gamma (\alpha (k+1))} \left( s^{-\nu\alpha} \mat{T} \right)^k \right)\vect{t}} ,{\quad s\geq 0}, \label{eq:LPT-MMLnu}
 \end{equation}
 where the series expansion relates to a generalized Wright hypergeometric function ({cf.} with  \cite[p.364]{ML-book-2014} for further details).
The similarity with the Laplace transform for $Y\sim \mbox{MML}(\alpha,\vect{\pi},\mat{T})$ may be appreciated by rewriting
 \begin{equation}
  L_Y(s;{\alpha,\vect{\pi},\mat{T}})=\vect{\pi}(s^\alpha\mat{I} - \mat{T})^{-1} \vect{t} = s^{-\alpha}\vect{\pi}(\mat{I}-s^{-\alpha}\mat{T})^{-1}\vect{t} ,{\quad s\geq 0},\label{eq:LPT-MML1} 
  \end{equation}
where we also notice that \eqref{eq:LPT-MMLnu} reduces to \eqref{eq:LPT-MML1} for $\nu=1$.

\section{Generalized matrix Mittag--Leffler distributions}\label{sec:GMML}
The convolution of Mittag--Leffler distributions is not a Mittag--Leffler distribution. However, if the components in the convolution have the same tail index, then the resulting distribution is a MML.
\begin{theorem}\label{th:convolution-MML}
Suppose that
 $ X\sim \mbox{MML}(\alpha,\vect{\pi}_1,\mat{T}_1)$ and $ Y\sim \mbox{MML}(\alpha,\vect{\pi}_2,\mat{T}_2)$.
 Then
\[ X+Y \sim \mbox{MML}(\alpha,\vect{\pi},\mat{T}) , \]
with
\[ \vect{\pi}=(\vect{\pi}_1,\vect{0}) \ \ \ \ \mbox{and} \ \ \ \ \mat{T} =
\begin{pmatrix}
  \mat{T}_1 & \vect{t}_1\vect{\pi_2} \\
  \mat{0} & \mat{T}_2
  \end{pmatrix}  .\]
\end{theorem}
 \begin{proof}
 This result follows from the Laplace transform of $X+Y$ being
 \begin{eqnarray*}
 L_{X+Y}(u;{ \alpha,\vect{\pi},\mat{T}})&=& \vect{\pi}_1(u^\alpha \mat{I}- \mat{T}_1)^{-1}\vect{t}_1\vect{\pi}_2(u^\alpha \mat{I}- \mat{T}_2)^{-1}\vect{t}_2 \\
 &=& (\vect{\pi}_1,\vect{0}) 
 \begin{pmatrix}
 (u^\alpha \mat{I} - \mat{T}_1)^{-1} & -(u^\alpha \mat{I} - \mat{T}_1)^{-1} (-\mat{t}_1\vect{\pi}_2)(u^\alpha \mat{I} - \mat{T}_2)^{-1} \\
 \mat{0} & (u^\alpha \mat{I} - \mat{T}_2)^{-1} \\
 \end{pmatrix}
 \begin{pmatrix}
 \vect{0} \\
 \vect{t}_2
 \end{pmatrix}\\
 &=&(\vect{\pi}_1,\vect{0}) 
 \left( u^\alpha \mat{I} -
 \begin{pmatrix}
  \mat{T}_1 & \mat{t}_1\vect{\pi}_2 \\
 \mat{0} & \mat{T}_2 \\
 \end{pmatrix}
 \right)^{-1}
 \begin{pmatrix}
 \vect{0} \\
 \vect{t}_2
 \end{pmatrix} .
 \end{eqnarray*}
 \end{proof}
Since $ X\sim \mbox{MML}(\alpha,\vect{\pi}_1,\mat{T}_1)$ implies that $cX \sim \mbox{MML}(\alpha,\vect{\pi},\mat{T})$ for any constant $c>0$, where 
  \[  \vect{\pi}=\vect{\pi}_1 \ \ \mbox{and} \ \ \mat{T}=c^{-\alpha}\mat{T}_1  , \]
we conclude that if $X_1,X_2,...,X_n$ are independent MML with the {\it same} tail index $\alpha$, then any linear combination $c_1X_1+...+c_nX_n$ with $c_1,c_2,...,c_n\geq 0$ is again MML with tail index $\alpha$.

The convolution of MML distributions with different tail indices are not MML distributions, but naturally lead to an extended class of MML distributions which we refer to as {\it Generalized} MML, as we will define in the sequel.
If $ X\sim \mbox{MML}(\alpha,\vect{\pi}_1,\mat{T}_1)$ and $ Y\sim \mbox{MML}(\beta,\vect{\pi}_2,\mat{T}_2)$ with $\alpha\neq \beta$, then calculations similar to the proof of Theorem \ref{th:convolution-MML} lead to $X+Y$ having Laplace transform
\begin{equation}
 L_{X+Y}(u) =(\vect{\pi}_1,\vect{0}) 
\left( \mat{\Delta}(u^\alpha \mat{I},u^\beta \mat{I}) -
\begin{pmatrix}
 \mat{T}_1 & \mat{t}_1\vect{\pi}_2 \\
\mat{0} & \mat{T}_2 \\
\end{pmatrix}
\right)^{-1}
\begin{pmatrix}
\vect{0} \\
\vect{t}_2
\end{pmatrix} , \label{eq:gen-MML}
\end{equation}
where $\mat{\Delta}(\mat{A},\mat{B})$ denotes the block diagonal matrix 
\[ \mat{\Delta}(\mat{A},\mat{B}) = \begin{pmatrix}
\mat{A} & \mat{0} \\
\mat{0} & \mat{B}
\end{pmatrix}  \]
for square matrices $\mat{A}$ and $\mat{B}$.
The linear combination $c_1X+c_2Y$ will then have a Laplace transform on the form,
\[  L_{c_1X+c_2Y}(u)=
(\vect{\pi}_1,\vect{0}) 
\left( \mat{\Delta}(u^\alpha \mat{I},u^\beta \mat{I}) -
\begin{pmatrix}
 c_1^{-\alpha}\mat{T}_1 & c_1^{-\alpha}\mat{t}_1\vect{\pi}_2 \\
\mat{0} & c_2^{-\beta}\mat{T}_2 \\
\end{pmatrix}
\right)^{-1}
\begin{pmatrix}
\vect{0} \\
c_2^{-\beta}\vect{t}_2
\end{pmatrix} .
 \]
This motivates the following definition. 

 \begin{definition}\normalfont
 A random variable $X$ is said to have a (univariate) generalized matrix Mittag--Leffler distribution, if there exist $\alpha_1,...,\alpha_n$ with $0<\alpha_i\leq 1$, and a phase--type representation $(\vect{\pi},\mat{T})$ for which the absolutely continuous part of its Laplace transform is given by
 \[  L^{\text{cont}}_X(u;{\vect{\alpha},\vect{\pi},\mat{T}}) = \vect{\pi}(\mat{\Delta}(u^{\alpha_1} \mat{I}_1,...,u^{\alpha_n} \mat{I}_n)-\mat{T})^{-1}\vect{t} ,  \quad{ u\geq 0}, \]
 where $\mat{I}_k$ are identity matrices and $\mbox{dim}(\mat{I}_1)+...+\mbox{dim}(\mat{I}_n)=\mbox{dim}(\mat{T})$. We write
 \[  X\sim \mbox{GMML}(\vect{\alpha},\vect{\pi},\mat{T}) , \]
 where $\vect{\alpha}=(\alpha_1,...,\alpha_n){\in\mathbb{R}_+^n}$.\end{definition}

Then, if $X_1,...,X_n$ are independent with 
\[ X_i\sim \mbox{GMML}(\vect{\alpha}_i,\vect{\pi}_i,\mat{T}_i) ,
\]
we get 
\[ X_1+...+X_n \sim  \mbox{GMML}(\vect{\alpha},\vect{\pi},\mat{T}) \]
where
\[  \vect{\alpha}=(\vect{\alpha}_1,...,\vect{\alpha}_n) , \]
\[  \vect{\pi}=(\vect{\pi}_1,\vect{0},...,\vect{0}) , \]
and
\[  \mat{T} =
\begin{pmatrix}
\mat{T}_1 & \vect{t}_1\vect{\pi}_2 & \mat{0} & ... & \mat{0} \\
\mat{0} & \mat{T}_2 & \vect{t}_2\vect{\pi}_3 & ... & \mat{0} \\
\mat{0} & \mat{0} & \mat{T}_3 & ... & \mat{0} \\
\vdots & \vdots &\vdots & \vdots\vdots\vdots & \vdots \\
\mat{0} &\mat{0} &\mat{0} &\vdots\vdots\vdots & \mat{T}_n
\end{pmatrix} .
  \]
By scaling, any non--negative non--zero linear combination of GMML distributed random variables will again follow a GMML distribution.  

% Then apply to rational LPT, with the usual known form, then realize the MML is closed under concvolutions
% and linear combinations of ind. r.v., to motivate the general definition.

\section{The multivariate matrix Mittag--Leffler distribution}\label{sec4}
Motivated by Section \ref{sec:GMML}, we proceed now to define the multivariate MML in a similar way as their underlying multivariate phase--type distributions.
\begin{definition}\normalfont
A random vector $\vect{X}=(X_1,...,X_n)$ has a multivariate GMML distribution in the wide sense, if all non--negative non--vanishing linear combinations $c_1X_1+...+c_nX_n$ have a GMML distribution.
\end{definition}
As for MVPH distributions, this definition is not very practical from a constructive point of view, and we shall introduce a subclass inspired by \eqref{eq:Kulkarni}. To this end we first notice the following result.
\begin{lemma}\label{lemma:Bernstein-Widder}
Let $\phi (s_1,...,s_k)$ be a multi\-dimensional Laplace transform and
let $g_1(x)$$,...,$ $g_k(x)$ denote functions for which $-g_i$ are completely monotone. Then it follows that
\[  L(s_1,...,s_k)=\phi (g_1(s_1),...,g_k(s_k))  \]
is again a Laplace transform. 
\end{lemma}
\begin{proof}
This follows immediately from the multidimensional Bernstein--Widder theorem, see \cite[p.87]{Bochner2005}, which states that
a multivariate function $\phi (s_1,...,s_k)$ is a multi\-dimensional Laplace transform if and only if
it is infinitely often differentiable and 
\[ (-1)^{n_{1}+\cdots+n_{k}} \frac{\partial^{n_{1}+\ldots+n_{k}} \phi}{\partial s_{1}^{n_{1}} \ldots \partial s_{k}^{n_{k}}} \geq 0  \]
for all $n_1\geq 0,...,n_k\geq 0$. 
\end{proof}
From this we immediately get the following important result.
\begin{theorem}
Let $(\vect{\pi},\mat{T},\mat{R})$ be a representation for a multivariate PH distribution \eqref{eq:Kulkarni}. Then
the multidimensional function
\begin{equation}
  \phi (\vect{u}) = \vect{\pi}\left(  \mat{\Delta}(\mat{R}\vect{u}^{\vect{\alpha}} ) - \mat{T} \right)^{-1}\vect{t},\quad { \vect{u}\in \mathbb{R}^n_+}, \label{eq:joint-laplace}\end{equation}
with $\vect{u}^{\vect{\alpha}}=(u_1^{\alpha_1},...,u_n^{\alpha_n})$, is a multidimensional Laplace transform.
\end{theorem}
From Theorem \ref{th:projection-MPH} we now obtain the following.
\begin{theorem}\label{th:proj1}
Let $\vect{w}\geq \vect{0}$ denote a non--zero vector and let $\vect{X}=(X_1,...,X_n)$ have a distribution given by the joint Laplace transform \eqref{eq:joint-laplace} with all $\alpha_i=\alpha$. Decompose $(\vect{\pi},\mat{T})$ as in \eqref{eq:decompose} according to $\mat{R}\vect{w}^\alpha$.
Then the distribution of $\langle \vect{X},\vect{w}\rangle$ has an atom at zero of size ${q}=\boldsymbol{\pi}_{0}\left(\boldsymbol{I}-\left(-\boldsymbol{T}_{00}\right)^{-1} \boldsymbol{T}_{0+}\right) \boldsymbol{e}$, and a possibly defective absolute continuous part which is $\mbox{MML}(\alpha,\vect{\pi}_{\vect{w}^\alpha},\vect{T}_{\vect{w}^\alpha})$, where $(\vect{\pi}_{\vect{w}^\alpha},\vect{T}_{\vect{w}^\alpha})$ is given in Theorem \ref{th:projection-MPH}.
\end{theorem}
\begin{proof} The result follows from
\begin{eqnarray*}
\Exp \left( e^{-u\langle \vect{X},\vect{w}\rangle}  \right)&=&\Exp \left( e^{-\langle \vect{X},u\vect{w}\rangle}  \right)\\
&\stackrel{\eqref{eq:joint-laplace}}{=}&
\vect{\pi} \left(  \mat{\Delta}(\mat{R}u^\alpha \vect{w}^\alpha) -\mat{T} \right)^{-1}\vect{t} \\
&\stackrel{\eqref{eq:basic-id}}{=}& {q} + \vect{\pi}_{\vect{w}^\alpha}(u^\alpha \mat{I}-\mat{T}_{\vect{w}^\alpha})^{-1}\vect{t}_{\vect{w}^\alpha} .
\end{eqnarray*}
\end{proof}

For possibly distinct $\alpha_i$, we proceed as follows. 

\begin{theorem}\label{th:proj2}
Let $\vect{w}\geq \vect{0}$ denote a non--zero vector and let $\vect{X}=(X_1,...,X_n)$ be a random vector with joint Laplace transform \eqref{eq:joint-laplace}. Decompose $(\vect{\pi},\mat{T})$ as in \eqref{eq:decompose} according to
$\mat{R}\vect{w}^{\vect{\alpha}}$.
Then the distribution of $\langle \vect{X},\vect{w}\rangle$ has an atom at zero of size $p=\boldsymbol{\pi}_{0}\left(\boldsymbol{I}-\left(-\boldsymbol{T}_{00}\right)^{-1} \boldsymbol{T}_{0+}\right) \boldsymbol{e}$ and a possibly defective absolute continuous part which is $\mbox{GMML}(\vect{\alpha},\vect{\pi}_{\vect{w}^\alpha},\vect{T}_{\vect{w}^\alpha})$, where $(\vect{\pi}_{\vect{w}^\alpha},\vect{T}_{\vect{w}^\alpha})$ is given in Theorem \ref{th:projection-MPH}.
\end{theorem}
\begin{proof} 
We have that
\begin{eqnarray*}
\Exp \left( e^{-u\langle \vect{X},\vect{w}\rangle} \right)&=&
\Exp \left( e^{-\langle \vect{X},u\vect{w}\rangle} \right) \\
&=&\vect{\pi} \left( \mat{\Delta}(\mat{R}(u\vect{w})^{\vect{\alpha}})-\mat{T} \right)^{-1}\vect{t} \\
&=&\vect{\pi} \left( \mat{\Delta}(\mat{R}\vect{w}^{\vect{\alpha}})\mat{\Delta}(u^{\vect{\alpha}})-\mat{T} \right)^{-1}\vect{t},
\end{eqnarray*}
where $\mat{\Delta}(u^{\vect{\alpha}})=\mbox{diag}(u^{\alpha_1},...,u^{\alpha_n})$. Now splitting into blocks according to
$E_+$ and $E_0$, we see that
\begin{eqnarray*}
\vect{\pi} \left( \mat{\Delta}(\mat{R}\vect{w}^{\vect{\alpha}})\mat{\Delta}(u^{\vect{\alpha}})-\mat{T} \right)^{-1}\vect{t}
&=&\vect{\pi} 
\begin{pmatrix}
 \mat{\Delta}(\mat{R}\vect{w}^{\vect{\alpha}})_+\mat{\Delta}(u^{\vect{\alpha}})_+-\mat{T}_{++} & -\mat{T}_{+0} \\
 -\mat{T}_{0+} & -\mat{T}_{00}
 \end{pmatrix}^{-1}\vect{t}\\
 &=&(\vect{\pi}_+,\vect{\pi}_0)
\begin{pmatrix}
\mat{A}_{11} & \mat{A}_{12} \\
\mat{A}_{21} & \mat{A}_{22}
\end{pmatrix}
 \begin{pmatrix}
 \vect{t}_+ \\
 \vect{t}_0
 \end{pmatrix},
\end{eqnarray*}
where
\begin{eqnarray*}
\mat{A}_{11}&=&\left( \mat{\Delta}(\mat{R}\vect{w}^{\vect{\alpha}})_+\mat{\Delta}(u^{\vect{\alpha}})_+-\mat{T}_{++} - 
\mat{T}_{+0}(-\mat{T}_{00})^{-1}\mat{T}_{0+}
 \right)^{-1} \\
 &=& \left( \mat{\Delta}(u^{\vect{\alpha}})_+ - (\mat{\Delta}(\mat{R}\vect{w}^{\vect{\alpha}})_+)^{-1} \left[ \mat{T}_{++} + 
\mat{T}_{+0}(-\mat{T}_{00})^{-1}\mat{T}_{0+} \right]
 \right)^{-1} \mat{\Delta}(\mat{R}\vect{w}^{\vect{\alpha}})_+^{-1} \\
 &=& \left( \mat{\Delta}(u^{\vect{\alpha}})_+ - \mat{T}_{\vect{w}^{\vect{\alpha}}}
 \right)^{-1} \mat{\Delta}(\mat{R}\vect{w}^{\vect{\alpha}})_+^{-1}, \\
  & & \\
 \mat{A}_{12}&=&\left( \mat{\Delta}(u^{\vect{\alpha}})_+ - \mat{T}_{\vect{w}^{\vect{\alpha}}}
 \right)^{-1} \mat{\Delta}(\mat{R}\vect{w}^{\vect{\alpha}})_+^{-1}\mat{T}_{+0}(-\mat{T}_{00})^{-1}, \\
  & & \\
 \mat{A}_{21}&=& (-\mat{T}_{00})^{-1}\mat{T}_{0+} \left( \mat{\Delta}(u^{\vect{\alpha}})_+ - \mat{T}_{\vect{w}^{\vect{\alpha}}}
 \right)^{-1} \mat{\Delta}(\mat{R}\vect{w}^{\vect{\alpha}})_+^{-1}, \\
 & & \\
 \mat{A}_{22}&=&(-\mat{T}_{00})^{-1}\left( \mat{I} + \mat{T}_{0+}\left( \mat{\Delta}(u^{\vect{\alpha}})_+ - \mat{T}_{\vect{w}^{\vect{\alpha}}}
 \right)^{-1} \mat{\Delta}(\mat{R}\vect{w}^{\vect{\alpha}})_+^{-1}\mat{T}_{+0}(-\mat{T}_{00})^{-1} \right) .
\end{eqnarray*}
Then
\begin{eqnarray*}
\vect{\pi}_+\mat{A}_{11}+\vect{\pi}_0\mat{A}_{21}&=&
\vect{\pi}_{\vect{w}^\alpha}\left( \mat{\Delta}(u^{\vect{\alpha}})_+ - \mat{T}_{\vect{w}^{\vect{\alpha}}}
 \right)^{-1} \mat{\Delta}(\mat{R}\vect{w}^{\vect{\alpha}})_+^{-1}, \\
\vect{\pi}_+\mat{A}_{12}+\vect{\pi}_0\mat{A}_{22}&=&\vect{\pi}_0(-\mat{T}_{00})^{-1} + 
\vect{\pi}_{\vect{w}^{\vect{\alpha}}}\left( \mat{\Delta}(u^{\vect{\alpha}})_+ - \mat{T}_{\vect{w}^{\vect{\alpha}}}
 \right)^{-1} \mat{\Delta}(\mat{R}\vect{w}^{\vect{\alpha}})_+^{-1}\mat{T}_{+0}(-\mat{T}_{00})^{-1} .
\end{eqnarray*}
Now inserting
\[ \begin{pmatrix}
\vect{t}_+ \\
\vect{t}_0
\end{pmatrix} = 
-\mat{T}\vect{e} = \begin{pmatrix}
-\mat{T}_{++}\vect{e}-\mat{T}_{+0}\vect{e} \\
-\mat{T}_{0+}\vect{e}-\mat{T}_{00}\vect{e} 
\end{pmatrix},
  \]
we get 
\begin{eqnarray*}
\lefteqn{\left(\vect{\pi}_+\mat{A}_{11}+\vect{\pi}_0\mat{A}_{21}\right) \vect{t}_+ + 
\left(\vect{\pi}_+\mat{A}_{12}+\vect{\pi}_0\mat{A}_{22} \right)\vect{t}_0 }~~~~\\
&=&\vect{\pi}_0 (\mat{I} - (-\mat{T}_{00})^{-1}\mat{T}_{0+})\vect{e} + 
\vect{\pi}_{\vect{w}^{\vect{\alpha}}} \left( \mat{\Delta}(u^{\vect{\alpha}})_+ - \mat{T}_{\vect{w}^{\vect{\alpha}}}
 \right)^{-1}\vect{t}_{\vect{w}^{\vect{\alpha}}} \\
 &=& p + \vect{\pi}_{\vect{w}^{\vect{\alpha}}} \left( \mat{\Delta}(u^{\vect{\alpha}})_+ - \mat{T}_{\vect{w}^{\vect{\alpha}}}
 \right)^{-1}\vect{t}_{\vect{w}^{\vect{\alpha}}}
\end{eqnarray*}
with
\[ \vect{t}_{\vect{w}^{\vect{\alpha}}} = -\mat{T}_{\vect{w}^{\vect{\alpha}}} \vect{e}.  \]
\end{proof}
From the previous results we see that we have found a sub-class of multivariate matrix Mittag--Leffler distributions with explicit Laplace transform. This allows us to concentrate on this class, and to make the following definition.
\begin{definition}\normalfont 
Let $\vect{X}=(X_1,...,X_n)$ be a random vector. Then we say that $\vect{X}$ has a multivariate matrix generalized Mittag--Leffler distribution if it has joint Laplace transform given by \eqref{eq:joint-laplace}, and write
$$\vect{X}\sim \mbox{GMML}(\vect{\alpha},\vect{\pi},\mat{T},\mat{R}).$$
\end{definition}
The following result generalizes Theorem 3.6 of \citet{albrecher2019matrix} to the multivariate case. { In particular, it gives the probabilistic interpretation of the GMML class as a family of random vectors whose marginals are absorption times of randomly-scaled, time-inhomogeneous Markov processes. The dependence of the corresponding Markov processes arises from the fact that they are all generated according to a reward structure on an underlying common Markov jump process.}
%As for the univariate case, we have a decomposition of the multivariate MML random variables into products of scaled phase--type and stable random variables.
%\section{Decomposition and moments}
\begin{theorem}\label{repPHstable}
Let  $\vect{X}=(X_1,...,X_n)\sim \mbox{GMML}(\vect{\alpha},\vect{\pi},\mat{T},\mat{R})$. Then
\begin{align}\label{prodrepresentation}
\vect{X}\stackrel{d}{=}\vect{W^{1/\alpha}}\bullet \vect{S_\alpha},
\end{align}
where $\vect{W^{1/\alpha}}=(W_1^{1/\alpha_1},\dots,W_n^{1/\alpha_n})$ with $\vect{W}=(W_1,...,W_n)\sim \mbox{MPH}^*(\vect{\pi
},\mat{T},\mat{R})$ (see \eqref{eq:Kulkarni}), and where  $\vect{S_\alpha}=(S_{\alpha_1},\dots,S_{\alpha_n})$ is a vector of independent stable random variables, each with Laplace transform $\exp(-u^{\alpha_i})$. Here, $\bullet$ refers to component-wise multiplication of vectors.
\end{theorem}
\begin{proof}
We observe that
\begin{align*}
\E(\exp(-\langle\vect{u},\vect{W^{1/\alpha}}\bullet \vect{S_\alpha} \rangle))&=
\int_{\mathbb{R}_+^n} \E(\exp(-\langle\vect{u},\vect{w^{1/\alpha}}\bullet \vect{S_\alpha} \rangle)) \dd F_{\vect{W}}(\vect{w}) \\
&=
\int_{\mathbb{R}_+^n} \exp(-[u_1^{\alpha_1}w_1+\dots+u_n^{\alpha_n}w_n]) \dd F_{\vect{W}}(\vect{w}) \\
&=
\int_{\mathbb{R}_+^n} \exp(-\langle \vect{u^\alpha},\vect{w}\rangle) \dd F_{\vect{W}}(\vect{w}) \\
&=\vect{\pi}\left(  \mat{\Delta}(\mat{R}\vect{u}^{\vect{\alpha}} ) - \mat{T} \right)^{-1}\vect{t},
\end{align*}
which implies the desired representation.
\end{proof}
\begin{remark}\label{rem48}\normalfont
From representation \eqref{prodrepresentation}, we have that the marginals of any multivariate GMML distribution
are regularly varying with indices $\alpha_1,\ldots,\alpha_n$, all smaller than 1. Moreover, by the multivariate version of Breiman's lemma (cf. \citet{basrak}) and the fact that multivariate phase--type distributions have moments of all orders, it follows that the tail independence structure of the vector $\vect{S_\alpha}$ carries over to $\vect{X}$. That is, the multivariate GMML family introduced in this paper has (very) heavy-tailed GMML marginals, but is tail-independent. As mentioned in the introduction, application areas for such models are e.g.\ given in \citet{resnick02}.  
\end{remark} 

A consequence of $\alpha_i<1$ is that the mean does not exist. To alleviate this potential practical drawback, it was proposed in \citet{albrecher2019matrix} to consider power-transformed variables in the univariate case. In the same way, we propose the following definition.
\begin{definition}\normalfont
    Let $\vect{X}\sim \mbox{GMML}(\vect{\alpha},\vect{\pi},\mat{T},\mat{R})$. For $\vect{\nu}>\vect{0}$, we define
    \[ \vect{Y}=\vect{X^{1/\nu}}\sim \mbox{GMML}^{1/\nu}(\vect{\alpha},\vect{\pi},\mat{T},\mat{R}), \]
    and refer to it as the class of power multivariate MML distributions.
\end{definition}

\noindent Under the power transform, the class is in general no longer closed under linear combinations. For fixed $\vect{\alpha}$, however, it possesses the following denseness property (in contrast to distributions with Laplace transform \eqref{eq:joint-laplace}). Here `dense on $\mathbb{R}_+^n$' means dense in the sense of weak convergence among all distributions on $\mathbb{R}_+^n$.
\begin{theorem}\label{th410}
%Let $\mathcal{X}$ be the class of $\mbox{MML}(\vect{\alpha},\vect{\pi},\mat{T},\mat{R})$ variables, $\mathcal{X}^\ast_{\vect{\alpha}}$ the subclass of the $\mbox{MML}(\vect{\alpha},\vect{\pi},\mat{T},\mat{R},\vect{\nu})$ variables for a fixed $\vect{\alpha}$, and $\mathcal{X}_{\vect{\gamma}}$ the subclass of $\mbox{MML}(\vect{\alpha},\vect{\pi},\mat{T},\mat{R},\vect{\nu})$ with fixed marginal tail indices $\vect{\alpha}\bullet\vect{\nu}=\vect{\gamma}^{-1}>0$. Then
(i) The class of $\mbox{GMML}(\vect{\alpha},\vect{\pi},\mat{T},\mat{R})$ variables is dense on $\mathbb{R}_+^n$.\\
(ii) For any fixed  $\vect{\alpha}$, the class of $\mbox{GMML}^{1/\nu}(\vect{\alpha},\vect{\pi},\mat{T},\mat{R})$ variables is dense on $\mathbb{R}_+^n$.\\
(iii) For any fixed marginal tail indices $\vect{\alpha}\bullet\vect{\nu}=\vect{\gamma}^{-1}>0$, the class of\\ $\mbox{GMML}^{1/\nu}(\vect{\alpha},\vect{\pi},\mat{T},\mat{R})$ variables is dense on $\mathbb{R}_+^n$.
\end{theorem}
\begin{proof}
(i) The statement is evident by noticing that we may choose $\vect{\alpha}\equiv\vect{1}$ and recalling that the class of variables with Laplace transform \eqref{eq:Kulkarni} is dense on $\mathbb{R}_+^n$. 

(ii) Let $\vect{0}<\vect{\nu}_1<\vect{\nu}_2<\cdots$ be any increasing and (entry-wise) diverging sequence of vectors and $Y$ be an arbitrary random vector on $\mathbb{R}_+^n$. Let $\vect{S_\alpha}$ be as in Theorem \ref{repPHstable} and notice that $\vect{S_\alpha}^{1/\vect{\nu}_n}\to\vect{1}$. In particular $\vect{S_\alpha}^{1/\vect{\nu}_n}\stackrel{d}{\to}\vect{1}$. Moreover, we may choose an independent sequence of vectors $\vect{W}_n$ with Laplace transforms of the form \eqref{eq:Kulkarni} such that $\vect{W}_n^{1/\vect{\nu}_n}\stackrel{d}{\to}Y$. Applying the continuous mapping theorem, and by the characterization of Theorem \ref{repPHstable}, the statement follows.

(iii) Similar to the previous case, let $\vect{0}<\vect{\alpha}_1<\vect{\alpha}_2<\cdots$ be an increasing sequence of vectors, converging to $\vect{1}$, and set $\vect{\nu}_n=(\vect{\gamma}\bullet\vect{\alpha}_n)^{-1}$. With $\vect{S_\alpha}$ as in Theorem \ref{repPHstable} we have that $\vect{S}_{\vect{\alpha}_n}^{1/\vect{\nu}_n}\stackrel{d}{\to}\vect{1}$. Choosing an independent sequence of vectors $\vect{W}_n$ with Laplace transforms of the form \eqref{eq:Kulkarni} and with $\vect{W}_n^{1/\vect{\nu}_n}\stackrel{d}{\to}Y$, the proof is finished as before.
\end{proof}

\begin{remark}\normalfont
{The above result shows how several classes of multivariate Mittag-Leffler distributions and their power transforms are dense in the set of all distributions of the $n$-dimensional positive orthant. However, since we are dealing with a tail-independent model, the number of phases increases drastically when faced with the need to capture dependence above high thresholds. Heuristically, the tail dependence is only correctly modelled in the limit. This is in some way analogous to the fact that phase--type distributions are dense on all distributions on the positive real line, but they are all light-tailed (of exponential decay), and very large dimensions are needed for approximations of heavy-tailed distributions, cf.\ \citet{bladt2017matrix}.}
\end{remark}

\section{Special structures and examples}\label{sec:examples}
From the previous sections, it becomes clear that the tail behavior of the GMML class is determined by the parameters $\alpha_i$ (cf.\ Remark \ref{rem48}) and the dependence structure is mainly triggered by the {parameters of the reward matrix $ \mat{R}$}, as these determine joint contributions to the size of each component. The marginal behavior and overall shape in the body of the distribution is then finally implied by the structure of the phase-type components ($\vect{\pi},\mat{T}$). In particular, the dimension $p$ of the latter also determines the potential for possible multimodalities of the components. 
In fact, Theorem \ref{th410} on the denseness of $\mbox{GMML}^{1/\nu}$ distributions on ${\mathbb R}_+^n$ relies (implicitly in part (i)) on the possibility of having arbitrarily large dimension $p$, a flexibility that is needed for modelling multiple modes, as the latter can require many phases. However, due to the possibly complex interaction of all parameters, one can not uniquely assign the role of each of the parameters to achieve a particular distributional behavior or shape. Moreover, for arbitrary combinations of parameters it is not always possible to get an explicit expression for the density of a GMML distribution (a complication inherited from the phase-type distributions). \\

We now proceed to give an example of a subclass that, however, does allow an explicit form. To that end, consider the special structure \eqref{feed-forward_structure1} and \eqref{feed-forward_structure2} for $(\vect{\pi},\mat{T},\mat{R})$, which in the exponential case led to the density 
 \eqref{eq:dens-map},
 \[ f(x_1,...,x_n;{\vect{\pi},\mat{T},\mat{R}})= \vect{\pi}e^{\mat{C}_1x_1}\mat{D}_1 e^{\mat{C}_2x_2}\mat{D}_2\cdots \mat{D}_{n-1}e^{\mat{C}_n x_n}\mat{D}_n\vect{e}  . \]
This choice of $(\vect{\pi},\mat{T},\mat{R})$, when plugged into \eqref{eq:joint-laplace}, results in 
 the joint Laplace transform  of $\vect{X}\sim \mbox{GMML}(\vect{\alpha},\vect{\pi},\mat{T},\mat{R})$  
 \begin{equation}
  L_X(\vect{u};{\vect{\theta}})=\vect{\beta}
 \begin{pmatrix}
 u_1^{\alpha_1}\mat{I}-\mat{C}_1 & -\mat{D}_1 & \mat{0} & \cdots & \mat{0} \\
 \mat{0} & u_2^{\alpha_2}\mat{I}-\mat{C}_2 & -\mat{D}_2 & \cdots & \mat{0} \\
 \mat{0} & \mat{0} & u_3^{\alpha_3}\mat{I}-\mat{C}_3 & \cdots & \mat{0} \\
 \vdots & \vdots & \vdots & \vdots\vdots\vdots & \vdots \\
 \vect{0} & \vect{0} & \vect{0} & \cdots &  u_n^{\alpha_n}\mat{I}-\mat{C}_n
 \end{pmatrix}^{-1}\!\!\!\!
 \begin{pmatrix}
 \vect{0} \\
 \vect{0} \\
 \vect{0} \\
 \vdots \\
 \mat{D}_n\vect{e} 
 \end{pmatrix},\label{thi} 
 \end{equation}
{ where we now use the shorthand notation $\vect{\theta}=(\vect{\alpha},\vect{\pi},\mat{T},\mat{R})$.}
For the resulting class of GMML distributions we can derive joint and marginal density functions, but first we notice the following lemma.  

\begin{lemma}\label{lemma:int-of-ML}
\[  \int_0^\infty x^{\alpha-1}\mbox{E}_{\alpha,\alpha}(\mat{T}x^\alpha)\dd x = -\mat{T}^{-1}.  \]
\end{lemma}
\begin{proof}
Since $\lambda\rightarrow \lambda x^{\alpha-1} {E}_{\alpha,\alpha}(-\lambda x^{\alpha})$ is an analytic function, and a density as a function of $x$, we get that
\begin{eqnarray*}
 \int_0^\infty x^{\alpha-1}{E}_{\alpha,\alpha}(\mat{T}x^\alpha)\dd x&=&
  \int_0^\infty x^{\alpha-1}\frac{1}{2\pi\ii} \int_\gamma {E}_{\alpha,\alpha}(s x^\alpha)(s\mat{I}-\mat{T})^{-1}\dd s \dd x \\
  &=&\frac{1}{2\pi\ii} \int_\gamma \left( \int_0^\infty x^{\alpha-1} {E}_{\alpha,\alpha}(s x^\alpha) \dd x \right)  (s\mat{I}-\mat{T})^{-1}\dd s \\
  &=& \frac{1}{2\pi\ii} \int_\gamma(-s^{-1})(s\mat{I}-\mat{T})^{-1}\dd s \\
  &=&-\mat{T}^{-1} .
\end{eqnarray*}
\end{proof}
\begin{remark}\rm 
The matrix $\mat{U}=-\mat{T}^{-1}$ is the so--called Green matrix which has the following  probabilistic interpretation: The element $(i,j)$ of $\mat{U}$ is the expected time that the Markov jump process underlying a phase--type distribution with generator $\mat{T}$ spends in state $j$ (prior to absorption) given that it starts in state $i$. 
\end{remark}

The main result of this section is as follows.
  \begin{theorem}
  The Laplace transform \eqref{thi} can equivalently be written as
  \begin{equation}
    L_X(\vect{u};{\vect{\theta}}) = \vect{\pi}\left( \prod_{i=1}^n (u_i^{\alpha_1}\mat{I}-\mat{C}_i)^{-1}\mat{D}_i\right) \vect{e} \label{eq:LPT-MAP},\quad { \vect{u}\in \mathbb{R}_+^n}.
  \end{equation}
The corresponding joint density is given by
\begin{equation}
   f_X(x_1,...,x_n;{\vect{\theta}}) = \vect{\pi}\left( \prod_{i=1}^n x_i^{\alpha_i-1}E_{\alpha_i,\alpha_i}(\mat{C}_i x_i^{\alpha_i})\mat{D}_i\right)\vect{e},\quad { x_i>0, \:\: i=1,\dots,n} .\label{eq:MAP} \end{equation}
   For the $i$'th marginal distribution of $X_i$ we have 
   \[  X_i \sim \mbox{MML}(\alpha_i,\vect{\beta}_i,\mat{C}_i)  \]
   where
   \[  \vect{\beta}_i = \vect{\pi}\prod_{j=1}^{i-1}(-\mat{C}_j)^{-1}\mat{D}_j .  \]
  \end{theorem}
  \begin{proof}
  It is sufficient to prove the result for $n=2$.
  \eqref{eq:LPT-MAP} follows from the general block diagonal inversion formula
  \[  \begin{pmatrix}
  \mat{A} & -\mat{B} \\
  \mat{0} & \mat{C}
  \end{pmatrix}^{-1} = \begin{pmatrix}
  \mat{A}^{-1} & \mat{A}^{-1}\mat{B}\mat{C}^{-1} \\
  \mat{0} & \mat{C}^{-1}
  \end{pmatrix}  . \]
  Concerning \eqref{eq:MAP}, we have that
  \begin{eqnarray*}
  \lefteqn{\int_0^\infty\int_0^\infty e^{-s_1x_1-s_2x_2}\vect{\pi} x_1^{\alpha_1} E_{\alpha_1,\alpha_1}(\mat{C}_1 x_1^{\alpha_1} )\mat{D}_1  x_2^{\alpha_2} E_{\alpha_2,\alpha_2}(\mat{C}_2 x_2^{\alpha_2} ) \mat{D}_2\vect{e} \dd x_1 \dd x_2 }~~~~ \\
&=&\int_0^\infty  e^{-s_1x_1}   x_1^{\alpha_1} \vect{\pi}E_{\alpha_1,\alpha_1}(\mat{C}_1 x_1^{\alpha_1} )\dd x_1 \mat{D}_1  \int_0^\infty e^{-s_2x_2}x_2^{\alpha_2} E_{\alpha_2,\alpha_2}(\mat{C}_2 x_2^{\alpha_2} ) \mat{D}_2\vect{e} \dd x_2 \\
&=&\vect{\pi}(u_1^{\alpha_1}\mat{I}-\mat{C}_1)^{-1}\mat{D}_1 (u_2^{\alpha_2}\mat{I}-\mat{C}_2)^{-1}\mat{D}_2\vect{e} \\
&=&(\vect{\pi},\vect{0})
\begin{pmatrix}
u_1^{\alpha_1}\mat{I}-\mat{C}_1 & -\mat{D}_1 \\
\mat{0} & u_2^{\alpha_2}\mat{I}-\mat{C}_2
\end{pmatrix}^{-1}
\begin{pmatrix}
\vect{0} \\
\mat{D}_2\vect{e}
\end{pmatrix},
  \end{eqnarray*}
  which is of the form \eqref{eq:joint-laplace}.

  The result on the marginal distributions follow from Lemma \ref{lemma:int-of-ML} and by using that 
  $(\mat{C}_i+\mat{D}_i)\vect{e}=\vect{0}$, { implying} that $(-\mat{C}_i)^{-1}\mat{D}_i\vect{e}=\vect{e}$.
  \end{proof}
The previous result can be used in the construction of bivariate (or multivariate) Mittag--Leffler distributions of a reasonably general type. 
  \begin{example}[Bivariate Mittag--Leffler distribution]\label{ex:5.1}\rm  {\ }\\
In this example we construct a class of bivariate distributions with Mittag--Leffler distributed marginals. The starting point is the construction of a bivariate exponential distribution underlying the MML. For details on this construction we refer to 
Section 8.3.2 of \citet{bladt2017matrix}. { Let $m$ be a positive integer and}
\[   \boldsymbol{S}=\left(\begin{array}{cccccc}{-{m} \lambda} & {({m}-1) \lambda} & {0} & {\dots} & {0} & {0} \\ {0} & {-({m}-1) \lambda} & {({m}-2) \lambda} & {\dots} & {0} & {0} \\ {0} & {0} & {-({m}-2) \lambda} & {\dots} & {0} & {0} \\ {\vdots} & {\vdots} & {\vdots} & {\ddots}  {\ddots} & {\vdots} & {\vdots} \\ {0} & {0} & {0} & {\dots} & {-2 \lambda} & {\lambda} \\ {0} & {0} & {0}  & {\cdots} & {0} & {-\lambda}\end{array}\right) . \]
Then for any initial distribution $\vect{\pi}=(\pi_1,...,{\pi_m})$, the phase--type distribution $\mbox{PH}(\vect{\pi},\mat{S})$ is simply an exponential distribution with intensity $\lambda$. Similarly, if we let
\[   \tilde{\boldsymbol{S}}=\left(\begin{array}{cccccc}{-\mu} & {\mu} & {0} & {\ldots} & {0} & {0} \\ {0} & {-2 \mu} & {2 \mu} & {\ldots} & {0} & {0} \\ {0} & {0} & {-3 \mu} & {\ldots} & {0} & {0} \\ {\vdots} & {\vdots} & {\vdots} & {\ddots}{\ddots} & {\vdots} & {\vdots} \\ {0} & {0} & {0} & {\cdots} & {-({m}-1) \mu} & {({m}-1) \mu} \\ {0} & {0} & {0} & {\ldots} & {0} & {-{m} \mu}\end{array}\right)  \]
and $\tilde{\vect{\pi}} = \frac{1}{{m}}\vect{e} =\left( \frac{1}{{m}},...,\frac{1}{{m}}  \right)$, then $\mbox{PH}(\tilde{\vect{\pi}},\tilde{\mat{S}})$ is again exponentially distributed with intensity $\mu$. Let $\mat{P}$ be a doubly stochastic matrix, i.e. its elements are non--negative and
\[  \mat{P}\vect{e}=\vect{e} \ \ \mbox{and} \ \ \vect{e}^\prime \mat{P}=\vect{e}^\prime , \]
and define
\[   \mat{T}  =  \begin{pmatrix}
\mat{S} & \lambda \mat{P}  \\
\mat{0} & \tilde{\mat{S}}
\end{pmatrix}  . \]
Consider the reward matrix 
\[  \mat{R}  = \begin{pmatrix}
\vect{e} & \vect{0} \\
\vect{0} & \vect{e}
\end{pmatrix} . \]
Then $\mbox{MPH}^*(\vect{e}_1^\prime,\mat{T},\mat{R})$ is a bivariate exponential distribution. This class of bivariate exponential distributions is capable of achieving any feasible correlation (ranging from $1-\pi^2/6$ to $1$) by choosing ${m}$ sufficiently large and $\mat{P}$ adequately (see \citet{biv1}). 
 Independence is achieved for
 \[  \mat{P} =  \frac{1}{{m}}\mat{E}, \]
 where $\mat{E}=\{ 1 \}_{i,j=1,...,{m}}$ is the matrix of ones, maximum negative (minimum) correlation (up to order ${m}$) by 
 \[  \mat{P}= \mat{I}  \]
 and maximum positive correlation for order up to ${m}$ by
 \[  \mat{P} = \{ \delta_{i,{m}-i+1} \}  , \]
 which is the anti--diagonal unit matrix, cf.\ \citet{Qi-Ming-biv-exp}. 

The correponding $\mbox{GMML}(\vect{\alpha},\vect{\pi},\mat{T},\mat{R})$ then has a density $f$ of the form
 \begin{equation}
  f(x_1,x_2;{\vect{\theta}}) = {m} \lambda \mu  x_1^{\alpha_1-1}x_2^{\alpha_2-1}  \vect{e}_1^\prime \mbox{E}_{\alpha_1,\alpha_1}(\mat{S}x_1^{\alpha_1})\mat{P}   \mbox{E}_{\alpha_2,\alpha_2}(\tilde{\mat{S}}x_2^{\alpha_2}) \vect{e}_n, \quad { x_1,x_2>0}, \label{eq:2-dim-gen}
  \end{equation}
where as usual $\vect{e}_i$ denotes the $i$'th Euclidian unit vector. The marginals are Mittag--Leffler distributions with densities
\[  f_{X_1}(x;{\alpha_1,\lambda}) = \lambda x^{\alpha_1-1}{E}_{\alpha_1,\alpha_1}(-\lambda x^{\alpha_1-1}) \ \ \mbox{and} \ \ f_{X_2}(x;{\alpha_2,\mu}) = \mu x^{\alpha_2-1}{E}_{\alpha_2,\alpha_2}(-\mu x^{\alpha_2-1}) , \]
{for $x>0$}, which follows directly from the invariance under different representations (parametrisations), or by simple integration and using Lemma \ref{lemma:int-of-ML}. Note that the present dependence structure has a very natural interpretation as a copula constructed in terms of combining marginal order statistics, cf.\ \citet{baker} and \cite[Sec.8.3.2]{bladt2017matrix}, here for Mittag-Leffler marginals. 

We can write the expression \eqref{eq:2-dim-gen} slightly more explicit. The eigenvalues of $\mat{S}$ are $-{m}\lambda$, $-({m}-1)\lambda$,..., $-\lambda$. To the eigenvalue $-\lambda {k}$ there corresponds an eigenvector $\vect{v}^{({k})}=(v_1^{({k})},...,v_n^{({k})})$ with
\begin{eqnarray*}
 v_1^{({k})}&=&1 \\
 v_{{i}+1}^{({k})}&=&\left( 1 - \frac{{k}-1}{{m}-{i}}  \right)v_{{i}}^{({k})}, \ \ {i}=1,...,{m}-1.
 \end{eqnarray*} 
 Similarly, $\tilde{\mat{S}}$ has eigenvalues $-\mu {m}, -\mu ({m}-1),...,-\mu$ and to the eigenvalue $-{k}\mu$ there corresponds an eigenvector $\vect{w}^{({k})}$ with
 \begin{eqnarray*}
 w_1^{({k})}&=&1 \\
 w_{{i}+1}^{({k})}&=&\left( 1 - \frac{{k}}{{i}}  \right)w_{{i}}^{({k})}, \ \ {i}=1,...,{m}-1.
 \end{eqnarray*} 
 Considering $\vect{v}^{({k})}$ and $\vect{w}^{({k})}$ as column vectors, we
 form the matrices $\mat{V}=(\vect{v}^{(1)},...,\vect{v}^{({m})})$ and 
 $\mat{W}=(\vect{w}^{(1)},...,\vect{w}^{({m})})$. Then we may write
 \begin{eqnarray*}
 {E}_{\alpha_1,\alpha_1}(\mat{S}x^{\alpha_1})&=&
 \mat{V} 
\mat{\Delta}\left( {E}_{\alpha_1,\alpha_1}(-{m}\lambda x^{\lambda_1}),..., {E}_{\alpha_1,\alpha_1}(-\lambda x^{\alpha_1})  \right)
   \mat{V}^{-1}, \\
   {E}_{\alpha_2,\alpha_2}(\tilde{\mat{S}}x^{\alpha_2})&=&\mat{W}
   \mat{\Delta}\left( {E}_{\alpha_2,\alpha_2}(-{m}\mu x^{\alpha_2}),..., {E}_{\alpha_1,\alpha_1}(-\mu x^{\alpha_2})  \right)
   \mat{W}^{-1} .
 \end{eqnarray*}
 Though the correlation between the Mittag--Leffler marginals is not defined (since moments of orders larger than $\alpha$ do not exist), some notion of dependence may be appreciated from the correlation structure of the underlying phase--type distribution.
 
In Figure \ref{bml.fig1} we depict a bivariate Mittag-Leffler density along with simulated data for the parameters  $\vect{\alpha}=(0.6,0.7)$, ${m}=20$, $\lambda=1$, $\mu=2$, and $\mat{P}$ the identity matrix.

In Figure \ref{bml.fig2} we use the same parameters but with $\mat{P}$ being the counter-identity matrix. As expected, the sign of the log-correlation is determined by the structure of the latter matrix. { Notice that the number of effective parameters corresponding to each of the two proposed structures is five.}

\begin{figure}[hh]
\centering
\includegraphics[width=7cm,trim=8cm 11cm 5cm .5cm,clip]{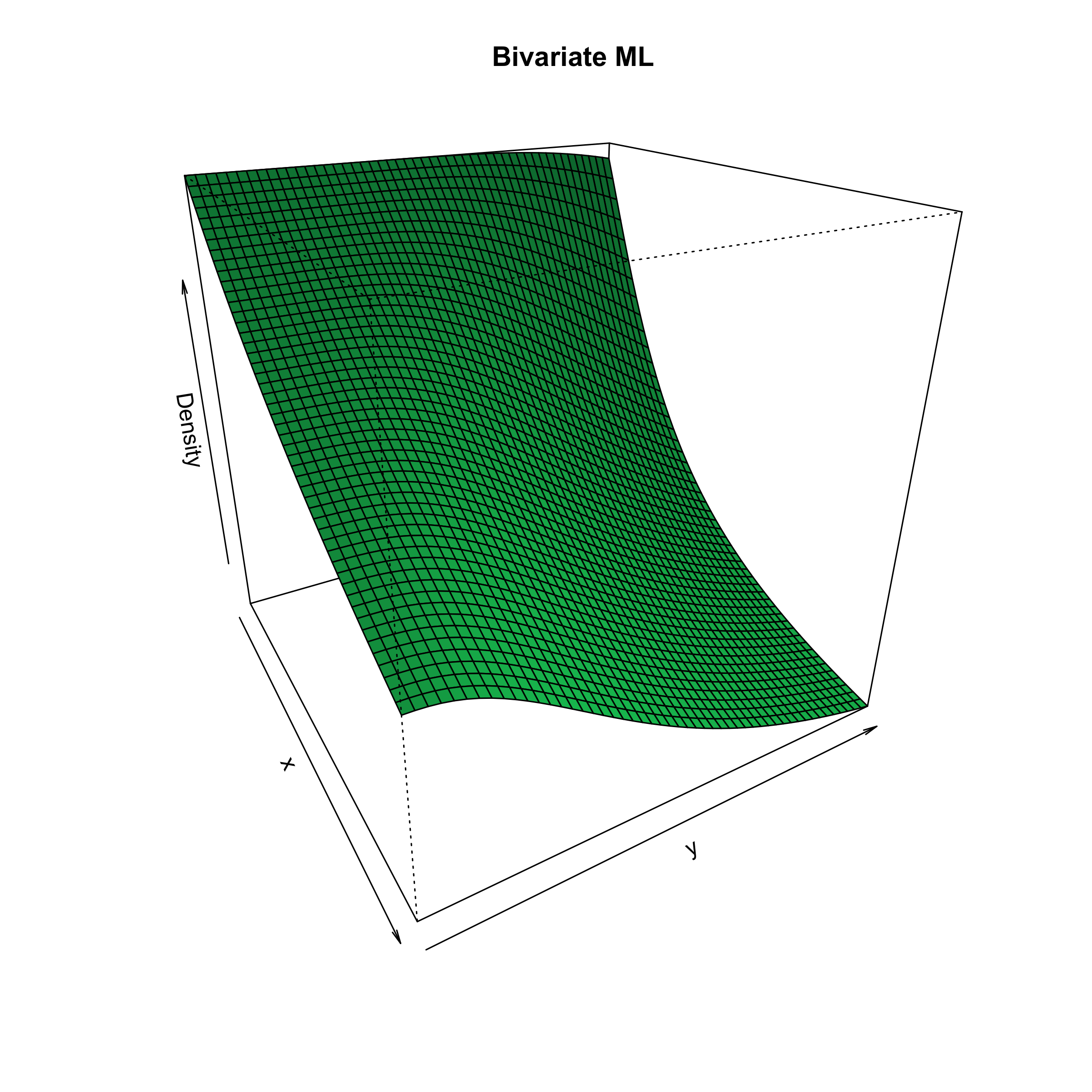}
\includegraphics[width=7cm,trim=.5cm .5cm .5cm .5cm,clip]{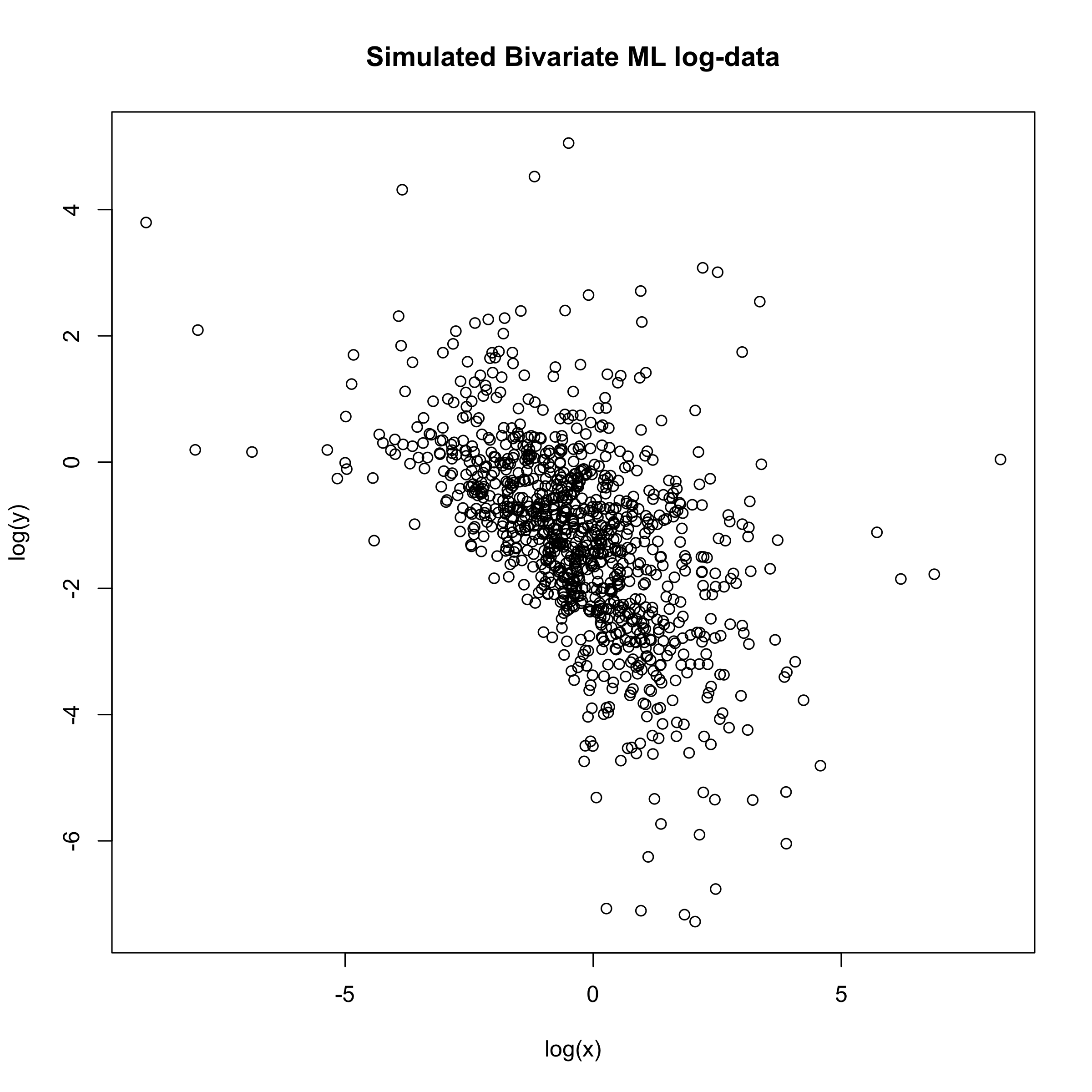}
\caption{Density and $1000$ simulated data-points from a bivariate ML distribution with negative log-correlation (empirical correlation of $-0.53$).} 
\label{bml.fig1}
\end{figure}

\begin{figure}[hh]
\centering
\includegraphics[width=7cm,trim=8cm 11cm 5cm .5cm,clip]{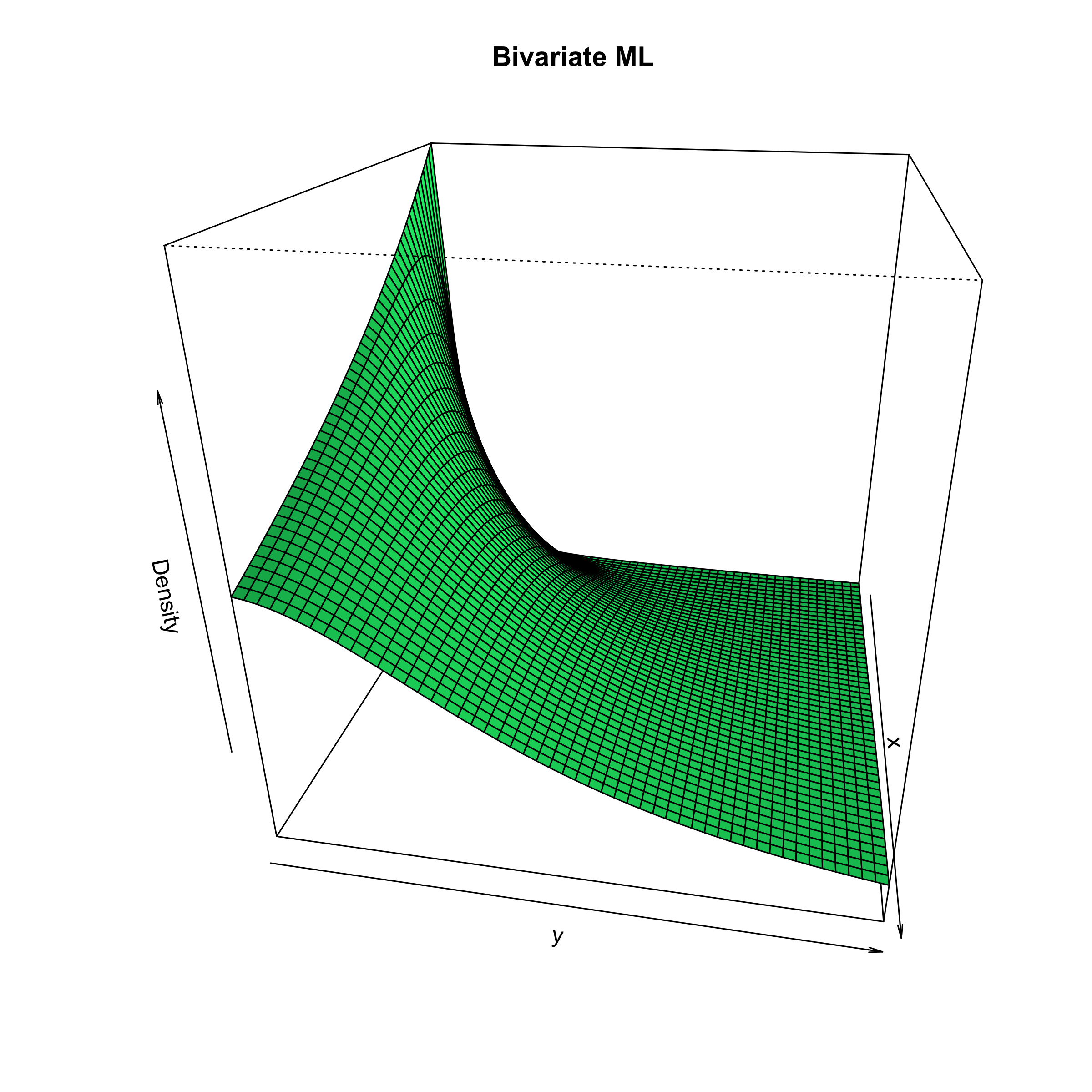}
\includegraphics[width=7cm,trim=.5cm .5cm .5cm .5cm,clip]{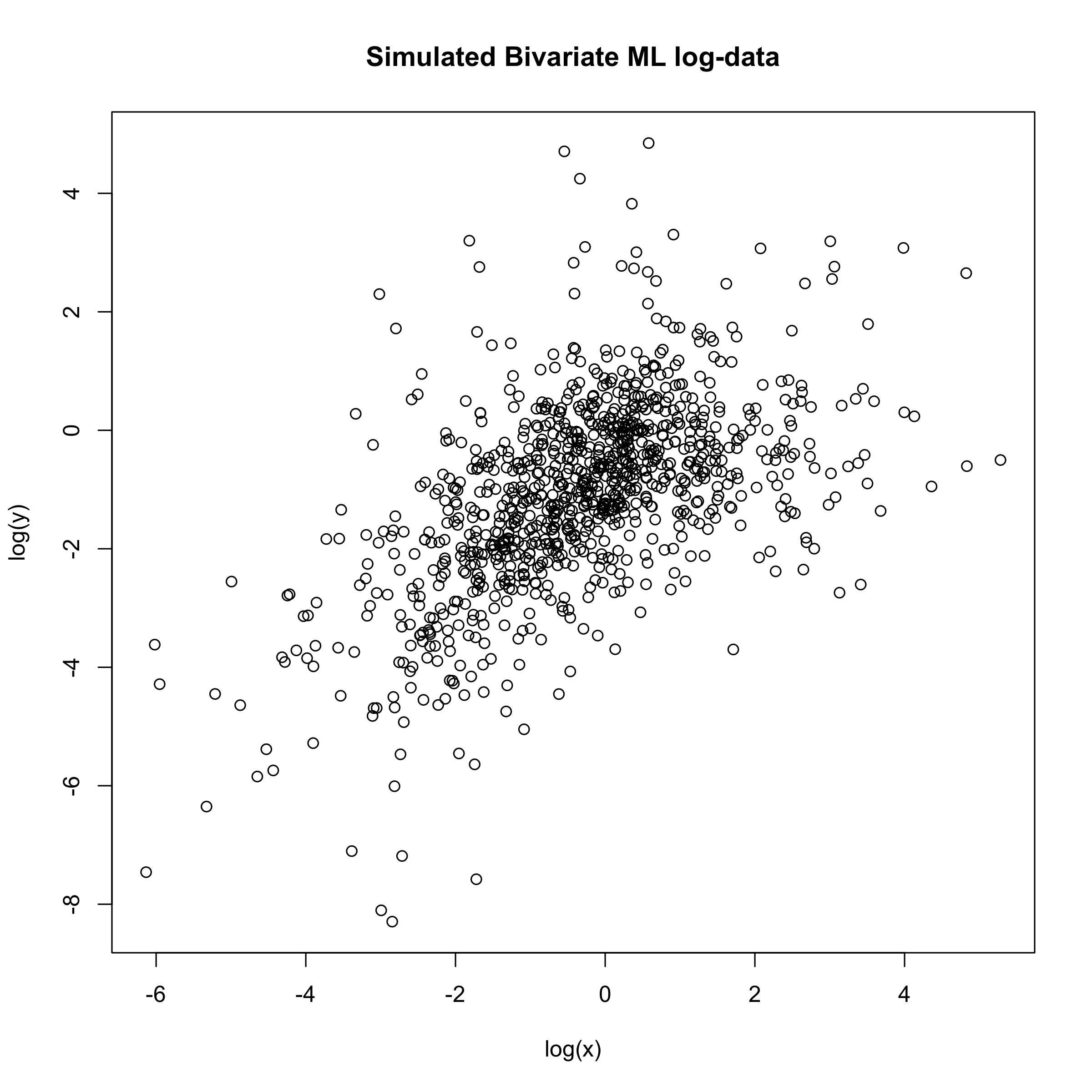}
\caption{Density and $1000$ simulated data-points from a bivariate ML distribution with positive correlation (empirical correlation of $0.55$).} 
\label{bml.fig2}
\end{figure}
 \qed
\end{example}

Concerning the power MML with this structure we have the following result. 
\begin{theorem}
Assume that $\vect{X}$ has joint density \eqref{eq:MAP}. Then $\vect{Y}=\vect{X}^{1/\vect{\nu}}$ has the joint density 
\[  f_Y(x_1,...,x_n;{\vect{\nu},\vect{\theta}}) = \vect{\pi}\left( \prod_{i=1}^n \nu_i x_i^{\alpha_i\nu_i-1}E_{\alpha_i,\alpha_i}(\mat{C}_i x_i^{\alpha_i\nu_i})\mat{D}_i\right)\vect{e} ,\quad {x_i>0,\:\: i=1,\dots ,n}, \]
and joint moments 
%\begin{eqnarray*}\lefteqn{
%\Exp \left( Y_1^{\theta_1}Y_2^{\theta_2}\cdots Y_n^{\theta_n}  \right)
%=\vect{\pi}
%\left( \prod_{i=1}^n \frac{\Gamma (1-\theta_i/(\nu_i \alpha_i))\Gamma(1+\theta_i/(\nu_i\alpha_i))}{\Gamma (1-\theta_i/\nu_i)}(-\mat{C}_i)^{-\theta_i/\nu_i\alpha_i-1} \mat{D}_i\right)\vect{e}}~~~~~ \\
%&=& \prod_{i=1}^n \left( \frac{\Gamma (1-\theta_i/(\nu_i \alpha_i))\Gamma(1+\theta_i/(\nu_i\alpha_i))}{\Gamma (1-\theta_i/\nu_i)}\right)\vect{\pi} \left( \prod_{i=1}^n (-\mat{C}_i)^{-\theta_i/\nu_i\alpha_i-1} \mat{D}_i\right)\vect{e} ,
%\end{eqnarray*}
\begin{multline*}
		\Exp \left( Y_1^{\theta_1}Y_2^{\theta_2}\cdots Y_n^{\theta_n}  \right)\\
	= \prod_{i=1}^n \left( \frac{\Gamma (1-\theta_i/(\nu_i \alpha_i))\Gamma(1+\theta_i/(\nu_i\alpha_i))}{\Gamma (1-\theta_i/\nu_i)}\right)\vect{\pi} \left( \prod_{i=1}^n (-\mat{C}_i)^{-\theta_i/\nu_i\alpha_i-1} \mat{D}_i\right)\vect{e} ,
\end{multline*}
where  $\nu_i\alpha_i >\theta_i>0$, for $i=1,2,...,n$.
\end{theorem} 
\begin{proof}
The form of the joint density is immediate. Concerning the form of the moments, it suffices to consider the case $n=2$. Using the decomposition \eqref{repPHstable}, we get 
\begin{eqnarray*}
 \Exp (Y_1^{\theta_1}Y_2^{\theta_2})&=& \Exp \left( W_1^{\frac{\theta_1}{\alpha_1 \nu_1}} W_2^{\frac{\theta_2}{\alpha_2 \nu_2}} S_{\alpha_1}^{\frac{\theta_1}{\nu_1}} S_{\alpha_2}^{\frac{\theta_2}{\nu_2}} \right)\\
&=&\Exp \left( W_1^{\frac{\theta_1}{\alpha_1 \nu_1}} W_2^{\frac{\theta_2}{\alpha_2 \nu_2}}  \right)\Exp \left( S_{\alpha_1}^{\frac{\theta_1}{\nu_1}}  \right) \Exp \left(S_{\alpha_2}^{\frac{\theta_2}{\nu_2}} \right) ,
  \end{eqnarray*}
where $(W_1,W_2)$ has a bivariate phase--type distribution with joint density \eqref{eq:dens-map}. Since 
\[ \Exp \left( S_{\alpha_i}^{\frac{\theta_i}{\nu_i}}  \right) = \frac{\Gamma \left( 1 -\frac{\theta_i}{\alpha_i\nu_1} \right)}{\Gamma \left( 1 -\frac{\theta_i}{\nu_1} \right)}, \]
the result then follows from Lemma \ref{lemma:MAP-structure}.
\end{proof}
\begin{example}\normalfont
Consider the case of a bivariate MML distribution,  $\theta_1=\theta_2=1$, $\nu_i\alpha_i>1$ and that $\mat{C}_1$ and $\mat{C}_2$ have the same dimension (the latter can always be achieved by augmenting the smaller one). Using the abbreviation
\[ c_i =  \frac{\Gamma (1-1/(\nu_i \alpha_i))\Gamma(1+1/(\nu_i\alpha_i))}{\Gamma (1-1/\nu_i)}, \ \ i=1,2, \]
we get
\begin{eqnarray*}
\Exp (Y_1)&=&c_1 \vect{\pi} (-\mat{C}_1)^{-1/(\alpha_1\nu_1)-1}\mat{D}_1\vect{e},\\
\Exp (Y_2)&=&c_2 \vect{\pi}(-\mat{C}_1)^{-1} \mat{D}_1 (-\mat{C}_2)^{-1/(\alpha_1\nu_1)-1}\mat{D}_2\vect{e},\\
\Exp (Y_1Y_2)&=&c_1c_2  \vect{\pi} (-\mat{C}_1)^{-1/(\alpha_1\nu_1)-1}\mat{D}_1(-\mat{C}_2)^{-1/(\alpha_2\nu_2)-1}\mat{D}_2 \vect{e}.
\end{eqnarray*}
%from which we may then calculate the covariance. 
If $\nu_i\alpha_i>2$ we can calculate variances and correlation. Indeed, with 
\[ c_i^\prime =  \frac{\Gamma (1-2/(\nu_i \alpha_i))\Gamma(1+2/(\nu_i\alpha_i))}{\Gamma (1-2/\nu_i)}, \ \ i=1,2, \]
one has
\begin{eqnarray*}
\Exp (Y_1^2)&=&c_1^\prime \vect{\pi}(-\mat{C}_1)^{-2/(\alpha_1\nu_1)-1}\mat{D}_1\vect{e}\\
\Exp (Y_2^2)&=&c_2^\prime \vect{\pi}(-\mat{C}_1^{-1}\mat{D}_1)(-\mat{C}_1)^{-2/(\alpha_2\nu_2)-1}\mat{D}_2\vect{e}
\end{eqnarray*}
from which the correlation coefficient is readily calculated.

In Figure \ref{pmml.fig1} we depict a bivariate density from a $\mbox{GMML}^{1/\vect{\nu}}(\vect{\alpha},\vect{\pi},\mat{T},\mat{R})$ distribution along with simulated data. The parameters are given by 
\begin{align*}
\vect{\alpha}=(0.6,0.7),\quad \vect{\beta}=\vect{\nu}\bullet \vect{\alpha}=(3,3),
\end{align*}
and the phase-type component being of the feed-forward structure \eqref{feed-forward_structure1} and \eqref{feed-forward_structure2}, with $n=2$, $\vect{\beta}_1=(1/3,1/3,1/3)$, $\vect{\beta}_2=\vect{0}$,
% \begin{align*}
% \mat{C}_1=\mat{C}_2=\begin{pmatrix}
%  -10 & 0 & 0  \\
% 0 &-1 &0 \\
%  0 & 0& -1/10
%  \end{pmatrix},
%  \end{align*} 
 {
 \[ \mat{C}_1=\mat{C}_2=\begin{pmatrix}
 -10 & 0 & 0  \\
0 &-1 &0 \\
 0 & 0& -1/10
 \end{pmatrix},\ \ \mbox{and} \ \   \mat{D}_1 = -\mat{C}_1 = \begin{pmatrix}
 10 & 0 & 0  \\
0 &1 &0 \\
 0 & 0& 1/10
 \end{pmatrix} .  \]
Hence both marginals are mixtures of power Mittag--Leffler distributions. The mixing probabilities of the two distributions are also the same, $(1/3,1/3,1/3)$, since the diagonal form of $\mat{D}_1$ ensures that the second mixture draws the same component as the first.
The first marginal mixture distribution has a density given by
  \begin{equation}
  f_1(x) =  \frac{5}{3} x^{3} \sum_{i=1}^3  \lambda_i  {E}_{0.6,0.6}(-\lambda_i x^{3}),\label{marginal1}
  \end{equation}
where $\lambda_1=10, \lambda_2=1$ and $\lambda_3=1/10$, while the second marginal density has the form
 \begin{equation}
   f_2(x) =  \frac{10}{7} x^{3} \sum_{i=1}^3  \lambda_i  {E}_{0.7,0.7}(-\lambda_i x^{3}) .
\label{marginal2}
   \end{equation} 
The reward matrix is
\[  \mat{R} = 
\begin{pmatrix}
 1 & 0 \\
 1 & 0 \\
 1 & 0 \\
 0 & 1 \\
 0 & 1 \\
 0 & 1
 \end{pmatrix} \]   
 and $Y_1$ and $Y_2$ simply correspond to the aforementioned mixtures. The structure of $\mat{D}_1$ implies a strong positive correlation. For example, if $Y_1$ is picked from the mixture component with rate $10$, then $Y_2$ will be picked from the same component (but then drawn independently).   
}

 In Figure \ref{pmml.fig2} we use the same parameters, except for 
\begin{align*}
\mat{D}_1=\begin{pmatrix}
 0 & 0 & 10  \\
0 &1 &0 \\
 1/10 & 0& 0
 \end{pmatrix}.
 \end{align*} 
 { Here the correlation between $Y_1$ and $Y_2$ will be negative: if $Y_i$ is drawn from the component with rate $10$, then $Y_j$ will be drawn from a component with rate $0.1$, $i\neq j$. The marginal distributions are again given by \eqref{marginal1} and \eqref{marginal2} since the mixing probabilities are all equal.}
 We observe how the sign of the correlation is affected by the structure of the matrix $\mat{D}_1$, and the fact that the matrices $\mat{C}_i$ are no longer of Erlang structure, the effect is qualitatively opposite to that of the bivariate ML case. One also sees that the class provides quite some flexibility in terms of the shape of the joint density function. 
 
\begin{remark}\normalfont
{ Dependence may often be constructed by introducing certain structures into the intensity matrices like in Example \ref{ex:5.1}. More generally, dependence between several random variables of MPH$^*$ type may be constructed using the so--called Baker copula (\citet{baker}), where order statistics are used and any feasible correlation structure can be obtained.}
\end{remark}

\begin{figure}[hh]
\centering{}
\includegraphics[width=7cm,trim=8cm 11cm 5cm .5cm,clip]{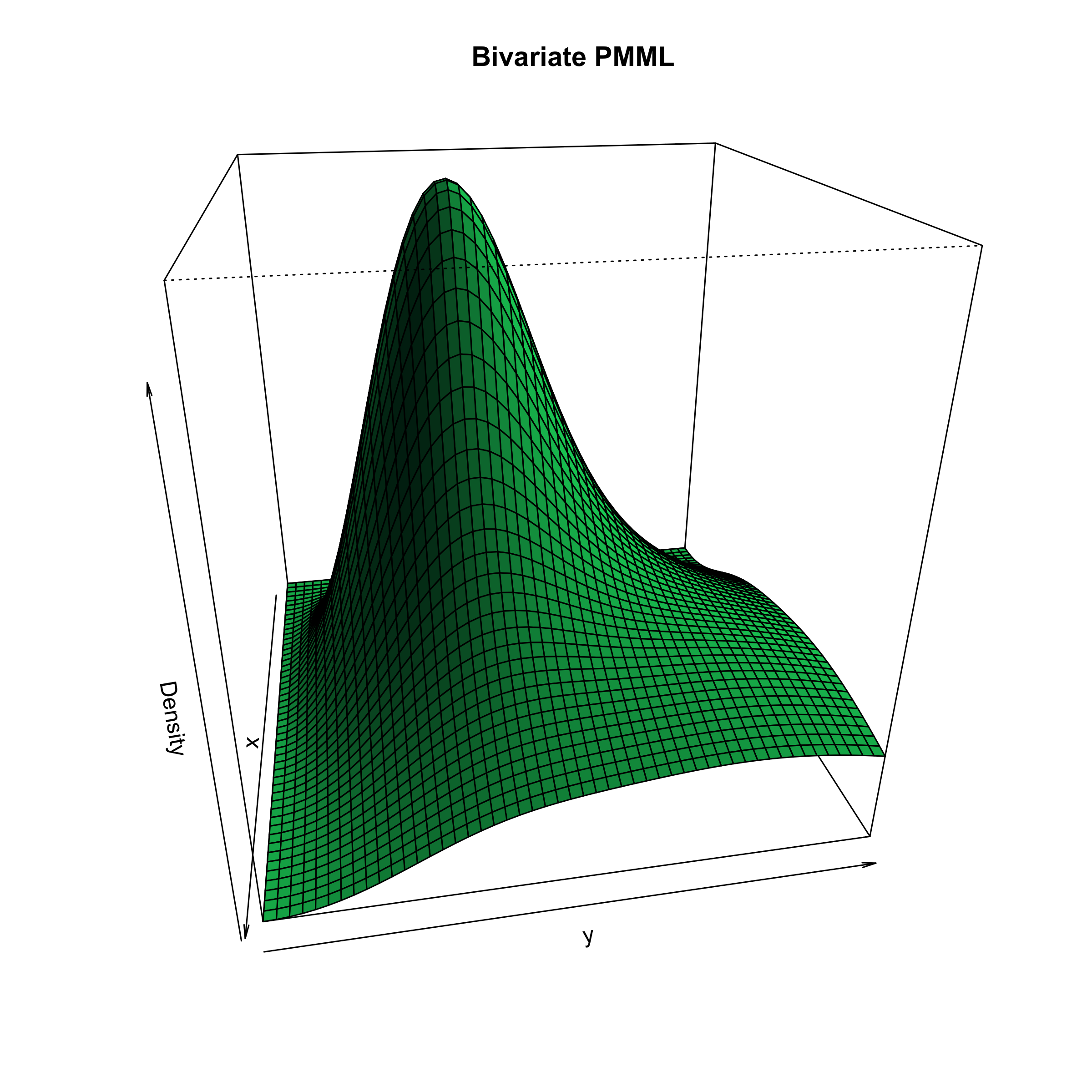}
\includegraphics[width=7cm,trim=.5cm .5cm .5cm .5cm,clip]{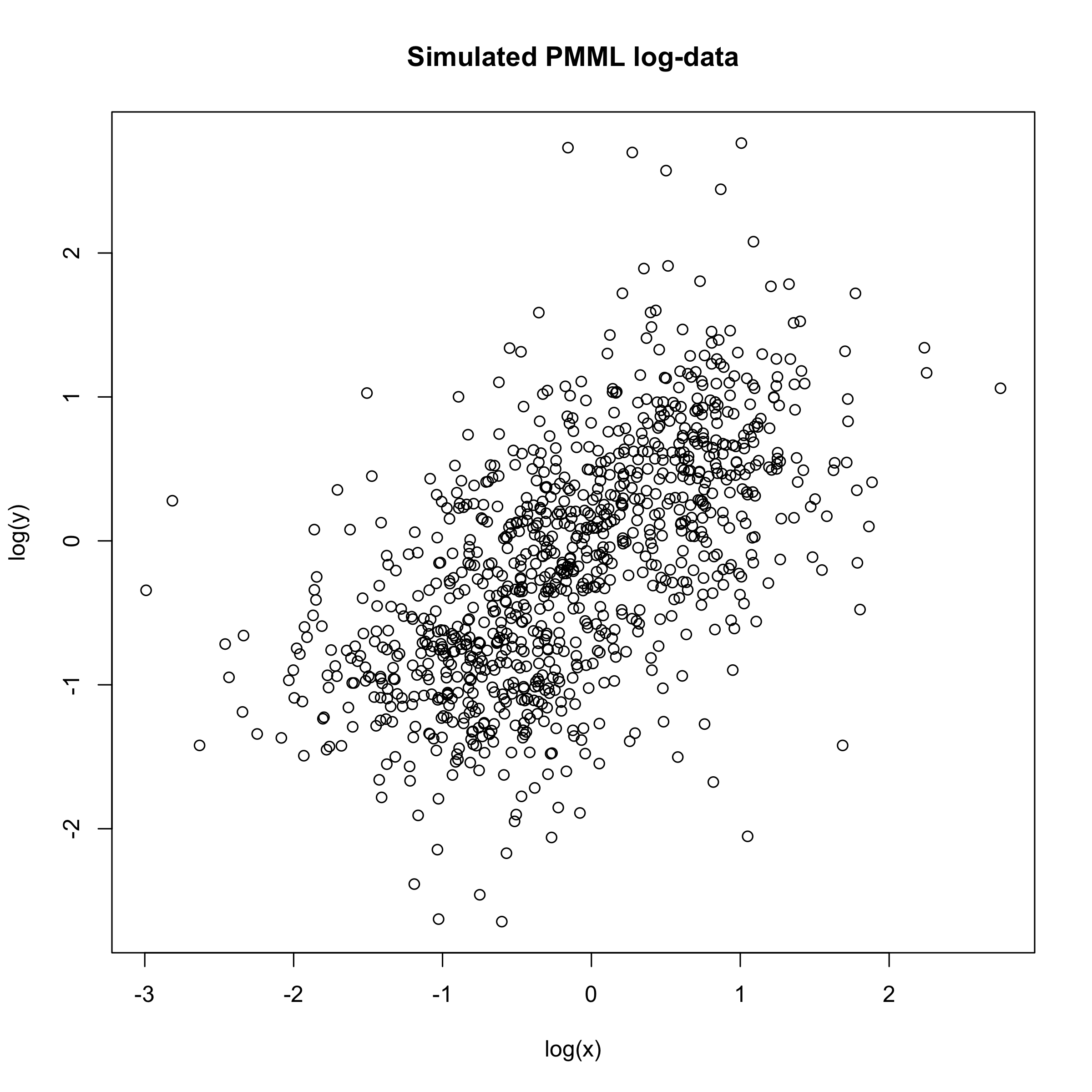}
\caption{Density and $1000$ simulated data-points from a power multivariate GMML distribution with positive correlation (true correlation of $0.35$ and empirical of $0.37$).} 
\label{pmml.fig1}
\end{figure}

\begin{figure}[hh]
\centering
\includegraphics[width=7cm,trim=8cm 11cm 5cm .5cm,clip]{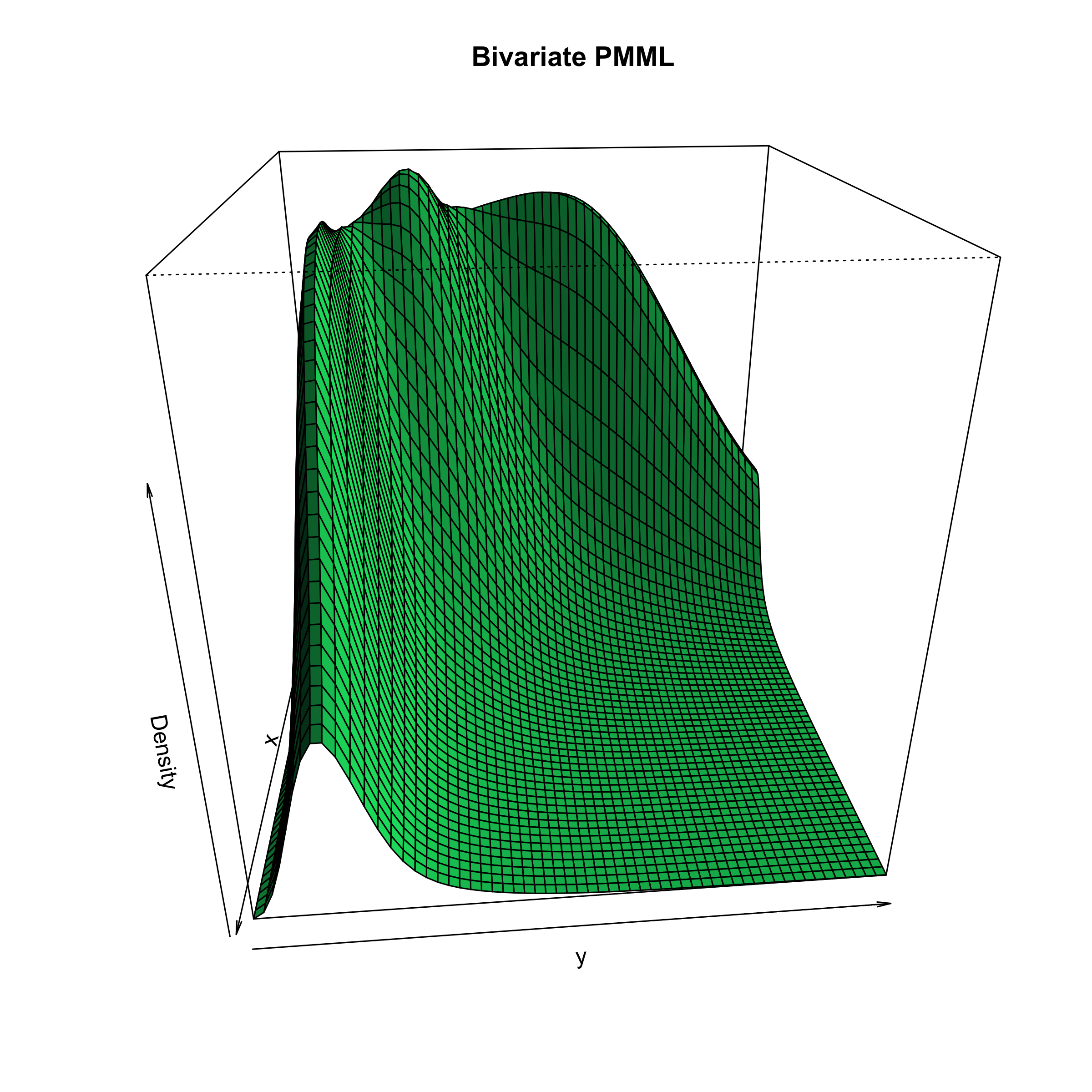}
\includegraphics[width=7cm,trim=.5cm .5cm .5cm .5cm,clip]{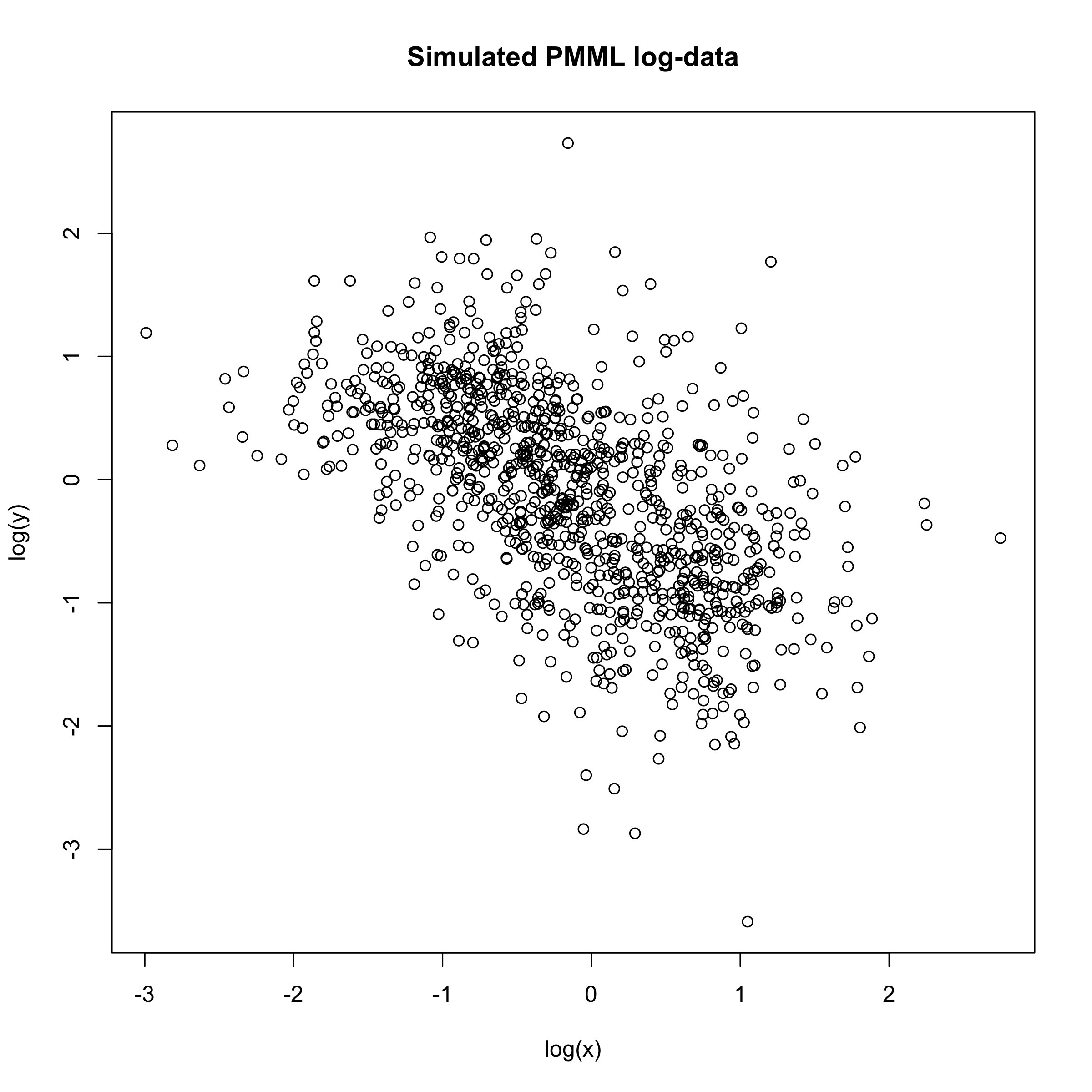}
\caption{Density and $1000$ simulated data-points from a power multivariate GMML distribution with negative correlation (true correlation of $-0.32$ and empirical of $-0.33$).} 
\label{pmml.fig2}
\end{figure}

\end{example}
\section{Conclusion}\label{sec:concl}
This paper introduces a class GMML of multivariate distributions with matrix Mittag-Leffler distributed marginals. With a construction essentially based on the multivariate phase--type distribution, the GMML class remains a flexible and tractable dense class of distributions maintaining a number of closed form properties. Two important sub--classes are considered, which lead to explicit formulas for distributional properties such as densities and fractional moments. 
{This makes it an attractive candidate for the modelling of both theoretical and practical aspects of multivariate heavy-tailed risks, in situations with  tail-independence. The present construction can not be extended to tail-dependent scenarios, so that other approaches will be needed for the latter, which will be an interesting topic for future research. }

  % We also illustrated concretely worked out examples of this new class of multivariate distributions. The approach taken in the analysis can in principle be used for the extension of any multivariate phase-type class to a dense multivariate setup with heavy tails, and it will be interesting in future research to apply this class of distributions for the modelling of multivariate risks in various fields of applications.  

% Carrying over analytic tools from the study of multivariate phase-type distributions to the present situation, this allowed to establish a number of properties for this flexible and tractable class with heavy-tailed marginals.
% We also show that subclasses of this new class are dense (in the sense of weak convergence) in the class of all distributions on the positive orthant, which makes it an attrative candidate for modelling of multivariate heavy-tailed risks.

\bibliography{Multi-MML_R2_b}
\bibliographystyle{apalike}
\setcitestyle{authoryear,open={(},close={)}}

\end{document}